\documentclass[10pt]{amsart}
      \usepackage{amsmath,amsfonts}
\def\rg{\hbox to 30pt{\rightarrowfill}}
\def\lg{\hbox to 30pt{\leftarrowfill}}

      \parskip\smallskipamount
          \newtheorem{theorem}{Theorem}[section]
      \newtheorem{definition}[theorem]{Definition}
      \newtheorem{proposition}[theorem]{Proposition}
      \newtheorem{corollary}[theorem]{Corollary}
      \newtheorem{lemma}[theorem]{Lemma}

      \newtheorem{remark}[theorem]{Remark}
      \makeatletter
      \@addtoreset{equation}{section}
      \makeatother

      \newcommand{\BB}{{\mathbb B}}
      \newcommand{\CC}{{\mathbb C}}
      \newcommand{\NN}{{\mathbb N}}
      
      \newcommand{\ZZ}{{\mathbb Z}}
      \newcommand{\DD}{{\mathbb D}}
      \newcommand{\RR}{{\mathbb R}}
      \newcommand{\FF}{{\mathbb F}}

      \newcommand{\cA}{{\mathcal A}}
      \newcommand{\cB}{{\mathcal B}}
      \newcommand{\cC}{{\mathcal C}}
      \newcommand{\cD}{{\mathcal D}}
      \newcommand{\cE}{{\mathcal E}}
      
      \newcommand{\cG}{{\mathcal G}}
      \newcommand{\cH}{{\mathcal H}}
      \newcommand{\cK}{{\mathcal K}}
      \newcommand{\cL}{{\mathcal L}}
      \newcommand{\cM}{{\mathcal M}}
      \newcommand{\cN}{{\mathcal N}}
      
      \newcommand{\cQ}{{\mathcal Q}}
      \newcommand{\cP}{{\mathcal P}}
      
      \newcommand{\cS}{{\mathcal S}}

      \newcommand{\cV}{{\mathcal V}}

      \newcommand{\supp}{\hbox{\rm{supp}}\,}
      \newcommand{\rank}{\hbox{\rm{rank}}\,}

      \newdimen\expt
      \def\boxit#1{\setbox0\hbox{$\displaystyle{#1}$}
            \hbox{\lower.4\expt
       \hbox{\lower3\expt\hbox{\lower\dp0
            \hbox{\vbox{\hrule height.4\expt
       \hbox{\vrule width.4\expt\hskip3\expt
            \vbox{\vskip3\expt\box0\vskip2\expt}%
       \hskip3\expt\vrule width.4\expt}\hrule height.4\expt}}}}}}
      
\begin{document}
       \pagestyle{myheadings}
      \markboth{ Gelu Popescu}{  Free Holomolomorphic functions on   polydomains  }

      \title [ Free Holomolomorphic functions on   polydomains   ]
      { Free Holomolomorphic functions on   polydomains }
        \author{Gelu Popescu}
     \date{\today}
\date{April 19, 2017}
      \thanks{Research supported in part by  NSF grant DMS 1500922}

      \address{Department of Mathematics, The University of Texas
      at San Antonio \\ San Antonio, TX 78249, USA}
      \email{\tt gelu.popescu@utsa.edu}

\begin{abstract} In this paper, we continue to develop the theory of free holomorphic functions on noncommutative regular polydomains. We find analogues of several classical results from complex analysis such as Abel theorem, Hadamard formula, Cauchy inequality, and Liouville theorem for entire functions, in our multivariable setting. We also provide a maximum principle and a Schwarz type lemma. These results are used to prove analogues of Weierstrass, Montel, and Vitali theorems for the algebra of free holomorphic functions on the regular polydomain, which turns out to be a complete metric space.
\end{abstract}

      \maketitle

\section*{Introduction}

 We developed  in  \cite{Po-von}, \cite{Po-disc}, \cite{Po-holomorphic}, \cite{Po-hyperbolic}, \cite{Po-hyperbolic2}, \cite{Po-holomorphic2}, and \cite{Po-automorphism}  a theory of free holomorphic functions on the  open unit ball
 $$
[B(\cH)^k]_1:=\left\{(X_1,\ldots, X_k)\in B(\cH)^k:\ \|X_1
X_1^*+\cdots +X_kX_k^* \|^{1/2} <1\right\},
$$
where  $k\in \NN:=\{1,2,\ldots\}$ and  $B(\cH)$ is the algebra
  of all bounded linear operators on a Hilbert space $\cH$.
Several classical results from complex analysis have analogues in this noncommutative multivariable setting. This theory was extended to noncommutative regular polyballs  ${\bf B_n}$ in a series of papers, starting with \cite{Po-poisson}, which led to our work on the curvature invariant \cite{Po-curvature-polyball}, the Euler characteristic \cite{Po-Euler-charact}, and  the  group of  free holomorphic automorphisms  on  ${\bf B_n}$ \cite{Po-automorphisms-polyball}. The regular  polyball ${\bf B_n}$, ${\bf n}=(n_1,\ldots, n_k)\in \NN^k$,  is a noncommutative analogue of the scalar polyball \, $(\CC^{n_1})_1\times\cdots \times  (\CC^{n_k})_1$   and  has a universal model $ {\bf S}:=\{{\bf S}_{i,j}\}$ consisting of   left creation
operators acting on the  tensor product $F^2(H_{n_1})\otimes \cdots \otimes F^2(H_{n_k})$ of full Fock spaces.

The goal of the present paper is to continue the development of the theory of free holomorphic functions on regular polydomains ${\bf D_q}$, which was initiated in
\cite{Po-domains}, \cite{Po-classification}, \cite{Po-Berezin3}, and \cite{Po-Berezin-poly}. This theory is related, via noncommutative Berezin transforms, to the study of operator algebras generated by the universal models associated with the regular polydomains, as well as to  the theory of functions in several complex variable (\cite{Kr}, \cite{Ru1}, \cite{Ru2}).
 We mention that the regular polydomain ${\bf D_q}$, ${\bf q}:=(q_1,\ldots, q_k)$, is a noncommutative analogue of the scalar polydomain
 $$\cD_{q_1}(\CC)\times \cdots \times \cD_{q_k}(\CC),
 $$
 where $\cD_{q_i}(\CC)\subset \CC^{n_i}$ is a domain generated by a positive regular polynomial $q_i\in \CC[Z_1,\ldots, Z_{n_i}]$ (see Section 1).
 We remark that, in general, one can  view the free holomorphic functions on noncommutative polydomains  as  noncommutative functions in the sense of \cite{KV}. Our approach is quite different and relies on  the universal models associated with the regular polydomains.

After a few preliminaries on  Berezin transforms on regular polydomains, we show, in Section 2, that the regular polydomain ${\bf D_q}$  is a noncommutative complete Reinhardt domain. We obtain characterizations for free holomorphic functions on regular polydomains and provide analogues of several results from complex analysis such as: Abel theorem, Hadamard formula, Cauchy inequality, and Liouville theorem for entire functions.

In Section 3, we prove a maximum principle for free holomorphic functions on regular polydomains and obtain a Schwarz lemma in this setting. In Section 4, we provide analogues of the classical results of Weierstrass, Montel, and Vitali (see \cite{Co}). These results are used to show that the algebra $Hol({\bf D_q})$ of all free holomorphic functions with complex coefficients on the polydomain ${\bf D_q}$ is a complete metric space with respect to an appropriate metric. 

We mention that  most of the results of this paper, including Weierstrass, Montel, and Vitali type theorems,  remain true for the radial part of the  polydomains ${\bf D_f^m}$, studied in \cite{Po-Berezin-poly}, and the proofs are essentially  the same.
We also remark that the results of this  paper   play an important role in our   project, in preparation,  concerning the free biholomorphic classification of polydomains and the associated polydomain algebras.

\bigskip

\section{Preliminaries on Berezin transforms on noncommutative polydomains}

A polynomial $q\in \CC\left<Z_1,\ldots, Z_n\right>$ in $n$ noncommuting indeterminates   is called {\it positive regular}  if all its coefficients are positive, the constant term is zero, and the coefficients of the linear terms $Z_1,\ldots, Z_n$ are different from zero. Let $\FF_{n}^+$ be the unital free semigroup on $n$ generators
$g_{1},\ldots, g_{n}$ and the identity $g_{0}$. We denote   $Z_{\alpha}:=Z_{j_1}\cdots Z_{j_p}$
  if   $\alpha=g_{j_1}\cdots g_{j_p}\in \FF_{n}^+$, and $Z_{g_0}=1$.

 If $A:=(A_1,\ldots, A_n)\in B(\cH)^n$ and $q=\sum_{\alpha} a_\alpha Z_\alpha$, $a_\alpha\in \CC$,  we define the map
$\Phi_{q,A}:B(\cH)\to B(\cH)$ by setting $\Phi_{q,A}(Y):=\sum_{\alpha} a_\alpha A_\alpha Y A_\alpha ^*$ for $Y\in B(\cH)$.
Given   ${\bf n}:=(n_1,\ldots, n_k)$  with $n_i\in  \NN:=\{1,2,\ldots\}$,  and a $k$-tuple ${\bf q}:=(q_1,\ldots, q_k)$ of positive regular polynomials $q_i\in \CC\left<Z_1,\ldots, Z_{n_i}\right>$, we associate with each  $k$-tuple  ${\bf X}:=(X_1,\ldots, X_k)\in B(\cH)^{n_1}\times\cdots \times B(\cH)^{n_k}$ the {\it defect mapping} ${\bf \Delta_{q,X}}:B(\cH)\to  B(\cH)$ defined by
$$
{\bf \Delta_{q,X}}:=\left(id -\Phi_{q_1, X_1}\right)\circ \cdots \circ\left(id -\Phi_{q_k, X_k}\right).
$$
We denote by $B(\cH)^{n_1}\times_c\cdots \times_c B(\cH)^{n_k}$
   the set of all tuples  ${\bf X}:=({ X}_1,\ldots, { X}_k)\in B(\cH)^{n_1}\times\cdots \times B(\cH)^{n_k}$, where ${ X}_i:=(X_{i,1},\ldots, X_{i,n_i})\in B(\cH)^{n_i}$, $i\in \{1,\ldots, k\}$,
     with the property that, for any $s,t\in \{1,\ldots, k\}$, $s\neq t$, the entries of ${ X}_s$ are commuting with the entries of ${ X}_t$. In this case we say that ${ X}_s$ and ${ X}_t$ are commuting tuples of operators. Note that the operators $X_{i,1},\ldots, X_{i,n_i}$ are not necessarily commuting.
Consider the {\it noncommutative regular polydomain}
$$
{\bf D_q^-}(\cH):=\left\{ {\bf X} \in B(\cH)^{n_1}\times_c\cdots \times_c B(\cH)^{n_k}: \ {\bf \Delta_{q,X}^p}(I)\geq 0 \ \text{ for any  }\   {\bf p}:=(p_1,\ldots, p_k)\in \{0,1\}^k\right\},
$$
where
 the {\it defect mapping} ${\bf \Delta_{X}^p}:B(\cH)\to  B(\cH)$ is defined by
$$
{\bf \Delta_{q, X}^p}:=\left(id -\Phi_{q_1,X_1}\right)^{p_1}\circ \cdots \circ\left(id -\Phi_{q_k X_k}\right)^{p_k}, \qquad  {\bf p}:=(p_1,\ldots, p_k)\in \ZZ_+^k,
$$
and we use the convention that $(id-\Phi_{q_i,X_i})^0=id$, the identity mapping on $B(\cH)$. The {\it abstract noncommutative polydomain} ${\bf D_q^-}$  is the disjoint union $\coprod_{\cH} {\bf D_q^-}(\cH)$, over all Hilbert spaces $\cH$.

 For each $i\in \{1,\ldots, k\}$, let $H_{n_i}$ be
an $n_i$-dimensional complex  Hilbert space with orthonormal basis
$e_1^i,\dots,e_{n_i}^i$.
  We consider the full Fock space  of $H_{n_i}$ defined by
$$F^2(H_{n_i}):=\bigoplus_{p\geq 0} H_{n_i}^{\otimes p},$$
where $H_{n_i}^{\otimes 0}:=\CC 1$ and $H_{n_i}^{\otimes p}$ is the
(Hilbert) tensor product of $p$ copies of $H_{n_i}$. Let $\FF_{n_i}^+$ be the unital free semigroup on $n_i$ generators
$g_{1}^i,\ldots, g_{n_i}^i$ and the identity $g_{0}^i$.
The length of $\alpha\in
\FF_{n_i}^+$ is defined by $|\alpha|:=0$ if $\alpha=g_0^i$  and
$|\alpha|:=p$ if
 $\alpha=g_{j_1}^i\cdots g_{j_p}^i$, where $j_1,\ldots, j_p\in \{1,\ldots, n_i\}$.
Set $e_\alpha^i :=
e^i_{j_1}\otimes \cdots \otimes e^i_{j_p}$ if
$\alpha=g^i_{j_1}\cdots g^i_{j_p}\in \FF_{n_i}^+$
 and $e^i_{g^i_0}:= 1\in \CC$.
It is  clear that $\{e^i_\alpha:\alpha\in\FF_{n_i}^+\}$ is an orthonormal
basis of $F^2(H_{n_i})$.

For each  $j\in\{1,\ldots,n_i\}$, we  define
 the {\it weighted left creation  operators} $W_{i,j}:F^2(H_{n_i})\to
F^2(H_{n_i})$,   associated with the abstract noncommutative
polydomain  ${\bf D}^-_{q_i} $    by setting
\begin{equation*}
W_{i,j} e_\alpha^i:=\frac {\sqrt{b_{i,\alpha} }}{\sqrt{b_{i,g_j
\alpha} }} e^i_{g_j \alpha}, \qquad \alpha\in \FF_{n_i}^+,
\end{equation*}
where
\begin{equation}\label{b-coef}
  b_{i,g_0^i} :=1\quad \text{ and } \quad
 b_{i,\alpha}:= \sum_{p=1}^{|\alpha|}
\sum_{{\gamma_1,\ldots,\gamma_p\in \FF_{n_i}^+}\atop{{\gamma_1\cdots \gamma_p=\alpha }\atop {|\gamma_1|\geq
1,\ldots, |\gamma_p|\geq 1}}} a_{i,\gamma_1}\cdots a_{i,\gamma_p}
\end{equation}
for all  $ \alpha\in \FF_{n_i}^+$ with  $|\alpha|\geq 1$. Due to the definition of $W_{i,j}$,  for each   $\beta_i\in \FF_{n_i}^+$, we have
\begin{equation*}
  W_{i,\beta_i}W_{i,\beta_i}^*
e^i_{\alpha_i} =\begin{cases} \frac
{ {b_{i,\gamma_i} }}{ {b_{i,\alpha_i} }}\,e^i_{\alpha_i}& \text{ if
}
\alpha_i=\beta_i\gamma_i, \ \, \gamma_i\in \FF_{n_i}^+ \\
0& \text{ otherwise. }
\end{cases}
\end{equation*}
Now, we
define the operator ${\bf W}_{i,j}$ acting on the tensor Hilbert space
$F^2(H_{n_1})\otimes\cdots\otimes F^2(H_{n_k})$ by setting
$${\bf W}_{i,j}:=\underbrace{I\otimes\cdots\otimes I}_{\text{${i-1}$
times}}\otimes W_{i,j}\otimes \underbrace{I\otimes\cdots\otimes
I}_{\text{${k-i}$ times}}.
$$
 We recall  (see \cite{Po-Berezin-poly}) that ,
 if  ${\bf W}_i:=({\bf W}_{i,1},\ldots,{\bf W}_{i,n_i})$, then

  $$(id-\Phi_{q_1,{\bf W}_1})\circ \cdots \circ(id-\Phi_{q_k,{\bf W}_k}) (I)={\bf P}_\CC,
   $$
   where ${\bf P}_\CC$ is the
 orthogonal projection from $\otimes_{i=1}^k F^2(H_{n_i})$ onto $\CC 1\subset \otimes_{i=1}^k F^2(H_{n_i})$, where $\CC 1$ is identified with $\CC 1\otimes\cdots \otimes \CC 1$.
     Moreover, ${\bf W}:=({\bf W}_1,\ldots, {\bf W}_k)$ is  a pure $k$-tuple  in the
noncommutative polydomain $ {\bf D_q^-}(\otimes_{i=1}^kF^2(H_{n_i}))$ and
 is called the {\it universal model} associated
  with the abstract noncommutative
  polydomain ${\bf D_q^-}$. We recall the definition of the {\it noncommutative Berezin kernel} associated with any element ${\bf X}:=({ X}_1,\ldots, { X}_k)\in {\bf D_q^-}(\cH)$ with $X_i:=(X_{i,1},\ldots, X_{i,n_i})$ as the operator
   $${\bf K_{q,X}}: \cH \to F^2(H_{n_1})\otimes \cdots \otimes  F^2(H_{n_k}) \otimes  \overline{{\bf \Delta_{q,X}}(I) (\cH)}$$
   defined by
   $$
   {\bf K_{q,X}}h:=\sum_{\beta_i\in \FF_{n_i}^+, i=1,\ldots,k}
   \sqrt{b_{1,\beta_1} }\cdots \sqrt{b_{k,\beta_k} }\,
   e^1_{\beta_1}\otimes \cdots \otimes  e^k_{\beta_k}\otimes {\bf \Delta_{q,X} }(I)^{1/2} X_{1,\beta_1}^*\cdots X_{k,\beta_k}^*h,
   $$
for $h\in \cH$, where the defect operator is defined by
$$
{\bf \Delta_{q,X} }(I)  :=(id-\Phi_{q_1,X_1})\circ\cdots \circ(id-\Phi_{q_k,X_k}) (I),
$$
and the coefficients  $b_{1,\beta_1}, \ldots, b_{k,\beta_k}$
are given above. Here, we  use the notation $X_{i,\alpha_i}:=X_{i,j_1}\cdots X_{i,j_p}$
  if  $\alpha_i=g_{j_1}^i\cdots g_{j_p}^i\in \FF_{n_i}^+$ and
   $X_{i,g_0^i}:=I$.
The noncommutative Berezin kernel associated with a $k$-tuple
${\bf X}:=({ X}_1,\ldots, { X}_k)$ in the noncommutative polydomain ${\bf D_q^-}(\cH)$ has the following properties.
\begin{enumerate}
\item[(i)] ${\bf K_{q,X}}$ is a contraction  and
$$
{\bf K_{q,X}^*}{\bf K_{q,X}}=
\lim_{m_k\to\infty}\ldots \lim_{m_1\to\infty}  (id-\Phi_{q_k,X_k}^{m_k})\circ\cdots \circ(id-\Phi_{q_1,X_1}^{m_1})(I),
$$
where the limits are in the weak  operator topology.
\item[(ii)]  If ${\bf X}$ is pure, i.e., for each $i\in \{1,\ldots, k\}$,   $ \lim_{m_i\to\infty}\Phi_{q_i,X_i}^{m_i}(I)=0$ in the weak operator topology,   then
$${\bf K_{q,X}^*}{\bf K_{q,X}}=I_\cH. $$
\item[(iii)]  For any $i\in \{1,\ldots, k\}$ and $j\in \{1,\ldots, n_i\}$,  $${\bf K_{q,X}} { X}^*_{i,j}= ({\bf W}_{i,j}^*\otimes I)  {\bf K_{q,X}}.
    $$
\end{enumerate}

Now, we introduce a class of  noncommutative Berezin  transforms associated  with regular polydomains.
If $A$ is a positive invertible operator, we write $A>0$.
Define the open polydomain ${\bf D_q}:=\coprod_{\cH} {\bf D_q}(\cH)$, where
$$
{\bf D_q}(\cH):=\left\{ {\bf X} \in B(\cH)^{n_1}\times_c\cdots \times_c B(\cH)^{n_k}: \ {\bf \Delta_{q,X}}(I)>0 \right\}.
$$
The {\it Berezin transform at} ${\bf X}\in {\bf D_q}(\cH)$,
 is the map $ \boldsymbol{\cB_{\bf X}}: B(\otimes_{i=1}^k F^2(H_{n_i}))\to B(\cH)$
 defined by
\begin{equation*}
 {\boldsymbol\cB_{\bf X}}[g]:= {\bf K^*_{q,\bf X}} (g\otimes I_\cH) {\bf K_{q,\bf X}},
 \qquad g\in B(\otimes_{i=1}^k F^2(H_{n_i})).
 \end{equation*}
 Let $\cP({\bf W})$  be the set of all polynomials $p({\bf W}_{i,j})$  in  the operators ${\bf W}_{i,j}$, $i\in \{1,\ldots, k\}$, $j\in \{1,\ldots, n_i\}$,  and the identity.
  If $g$ is in the operator space
  $$\cS:=\overline{\text{\rm  span}} \{ p({\bf W}_{i,j})q({\bf W}_{i,j})^*:\
p({\bf W}_{i,j}),q({\bf W}_{i,j}) \in  \cP({\bf W})\},
$$
where the closure is in the operator norm, we  define the Berezin transform at  ${\bf X}\in {\bf D_q^-}(\cH)$, by
  $${\boldsymbol\cB_{\bf X}}[g]:=\lim_{r\to 1} {\bf K}^*_{{\bf q},r{\bf X}} (g\otimes I_\cH) {\bf K}_{{\bf q},{r\bf X}},
 \qquad g\in  \cS,
 $$
 where the limit is in the operator norm topology.
In this case, the Berezin transform at ${\bf X}$ is a unital  completely positive linear  map such that
 $${\boldsymbol\cB_{\bf X}}({\bf W}_{\boldsymbol\alpha} {\bf W}_{\boldsymbol\beta}^*)={\bf X}_{\boldsymbol\alpha} {\bf X}_{\boldsymbol\beta}^*, \qquad \boldsymbol\alpha, \boldsymbol\beta \in \FF_{n_1}^+\times \cdots \times\FF_{n_k}^+,
 $$
 where  ${\bf W}_{\boldsymbol\alpha}:= {\bf W}_{1,\alpha_1}\cdots {\bf W}_{k,\alpha_k}$ if  $\boldsymbol\alpha:=(\alpha_1,\ldots, \alpha_k)\in \FF_{n_1}^+\times \cdots \times\FF_{n_k}^+$. The {\it polydomain algebra} $\boldsymbol\cA({\bf D_q})$ is the norm closed algebra generated by all  ${\bf W}_{i,j}$ and the identity, while the noncommutative Hardy algebra $F^\infty({\bf D_q})$ is the weakly closed algebra generated by ${\bf W}_{i,j}$ and the identity.
 The restriction of
noncommutative Berezin transform  ${\boldsymbol\cB_{\bf X}}$ to the polydomain algebra $\boldsymbol\cA({\bf D_q})$ is a
completely contractive homomorphism. If, in addition,
 ${\bf X}$ is a  pure $k$-tuple,
then
$$\lim_{r\to 1} {\bf B}_{r{\bf X}}[g]= {\bf B}_{\bf X}[g],\qquad g\in \cS.$$
The  Berezin transform will play an important role in this paper.
   More properties  concerning  noncommutative Berezin transforms and multivariable operator theory on noncommutative balls and  polydomains can be found in \cite{Po-poisson}, \cite{Po-holomorphic}, \cite{Po-holomorphic2}, \cite{Po-automorphism},    \cite{Po-Berezin3}, and \cite{Po-Berezin-poly}.
For basic results on completely positive (resp. bounded)  maps  we refer the reader to \cite{Pa-book} and \cite{Pi-book}.

\bigskip

\section{ Free holomorphic functions  on noncommutative   polydomains}

In this section, we show that the regular polydomain ${\bf D_q}$ is a noncommutative complete Reinhardt domain. We study free holomorphic functions on regular polydomains and provide analogues of several classical results from complex analysis such as: Abel theorem, Hadamard formula, Cauchy inequality,  and  Liouville theorem for entire functions.

 A subset $G$  of $ B(\cH)^{n_1}\times\cdots \times B(\cH)^{n_k}$
    is called {\it complete Reinhardt set}  if \, ${\bf zX}\in G$ for any ${\bf X}\in G$ and  ${\bf z}\in  \overline{\DD}^{n_1+\cdots+n_k}$, where ${\bf zX}:=\{z_{i,j} X_{i,j}\}$ if ${\bf X}:=\{X_{i,j}\}$ and ${\bf z}=\{z_{i,j}\}$ for $i\in \{1,\ldots, k\}$ and $j\in \{1,\ldots, n_i\}$.

\begin{proposition}\label{reg-poly} The  following statements hold:
\begin{enumerate}
\item [(i)]  The regular polydomain  ${\bf D_q}(\cH)$ is relatively  open in   $ B(\cH)^{n_1}\times_c\cdots \times_c B(\cH)^{n_k}$, and its closure ${\bf D_q}(\cH)^-$, in the operator norm topology,  coincides with ${\bf D_q^-}(\cH)$.
\item [(ii)] ${\bf D_q}(\cH)$is a complete Reinhardt domain such that
$$
{\bf D_q}(\cH)=\bigcup_{{\bf z}\in  \overline{\DD}^{n_1+\cdots+n_k}}{\bf z}{\bf D_q}(\cH)=\bigcup_{{\bf z}\in  {\DD}^{n_1+\cdots+n_k}}{\bf z}{\bf D_q}(\cH)^-=\bigcup_{{\bf z}\in  {\DD}^{n_1+\cdots+n_k}}{\bf z}{\bf D_q}(\cH).
$$
and $$ {\bf D_q}(\cH)=\bigcup_{0\leq r<1}r{\bf D_q}(\cH)=\bigcup_{0\leq r<1}r{\bf D_q}(\cH)^-.
$$
\item [(iii)] ${\bf D_q}(\cH)^-$ is a complete Reinhardt  set and
$$
{\bf D_q}(\cH)^-=\bigcup_{{\bf z}\in  \overline{\DD}^{n_1+\cdots+n_k}}{\bf z}{\bf D_q}(\cH)^- = \bigcup_{0\leq r\leq 1}r{\bf D_q}(\cH)^-.
$$
\end{enumerate}
\end{proposition}
\begin{proof}
 Fix ${\bf X}=(X_1,\ldots, X_k)\in {\bf D_q}(\cH)$ and let  $c>0$  be such that ${\bf \Delta_{q,X}}(I)>cI$.  If  $d\in (0,c)$, then there is $\epsilon>0$  such that
$
-dI\leq {\bf \Delta_{q,Y}}(I)-{\bf \Delta_{q,X}}(I)\leq dI
$
for any ${\bf Y}=(Y_1,\ldots, Y_k)\in B(\cH)^{n_1}\times_c\cdots \times_c B(\cH)^{n_k}$ with $\max_{i\in \{1,\ldots, k\}}\|X_i-Y_i\|<\epsilon$.
Hence,
$$
{\bf \Delta_{q,Y}}(I)=({\bf \Delta_{q,Y}}(I)-{\bf \Delta_{q,X}}(I))+{\bf \Delta_{q,X}}(I)\geq (c-d)I>0,
$$
which proves that $Y\in {\bf D_q}(\cH)$. Consequently,  ${\bf D_q}(\cH)$ is relatively  open in   $ B(\cH)^{n_1}\times_c\cdots \times_c B(\cH)^{n_k}$ with respect to the product topology. Now,  we prove that ${\bf D_q}(\cH)^-={\bf D_q^-}(\cH)$.
First, we show that if  $\lambda_i\in \overline{\DD}$, $i\in \{1,\ldots, k\}$,    and ${\bf W}=({\bf W}_1,\ldots, {\bf W}_k)$ is the universal model for the regular polydomain ${\bf D_q^-}$,  then
\begin{equation}
\label{qq}
(id-\Phi_{q_1,\lambda_1 {\bf W}_1})^{p_1}\circ\cdots \circ (id-\Phi_{q_k\lambda_k {\bf W}_k})^{p_k}(I)\geq \prod_{i=1}^k (1-|\lambda_i|^2)^{p_i}I,\qquad p_i\in \{0,1\}.
\end{equation}
We recall that two operators $A,B\in B(\cH)$ are called doubly commuting if $AB=BA$ and $AB^*=B^*A$. Since the entries of ${\bf W}_s=({\bf W}_{s,1},\ldots, {\bf W}_{s,n_s})$ are doubly commuting with the entries of ${\bf W}_t=({\bf W}_{t,1},\ldots, {\bf W}_{t,n_t})$, whenever $s,t\in\{1,\ldots, k\}$, $s\neq t$, we have
$$(id-\Phi_{q_1,\lambda_1 {\bf W}_1})^{p_1}\circ\cdots \circ (id-\Phi_{q_k,\lambda_k {\bf W}_k})^{p_k}(I)=
\prod_{i=1}^k (I-\Phi_{q_i,\lambda_i{\bf W}_i}(I))^{p_i},\qquad p_i\in \{0,1\}.
 $$
Since
 $I-\Phi_{q_i,\lambda_i{\bf W}_i}(I)\geq (1-|\lambda_i|^2)I$, the  inequality \eqref{qq} follows.
Taking into account that ${\bf D_q}(\cH)$ is open in the operator norm topology, it is clear that if  ${\bf X}\in {\bf D_q}(\cH)$, then there is $r\in [0,1)$ such that $\frac{1}{r}{\bf X}\in {\bf D_q}(\cH)$. Applying the Berezin transform at $\frac{1}{r}{\bf X}$ to the  inequality of \eqref{qq}, when $\lambda_i=r$ for any $i\in \{1,\ldots, k\}$, we deduce that
\begin{equation*}
{\bf \Delta_{q,X}^p}(I)=(id-\Phi_{q_1,X_1})^{p_1}\circ\cdots \circ (id-\Phi_{ q_k, {X}_k})^{p_k}(I)\geq \prod_{i=1}^k (1-r^2)^{p_i}I,\qquad p_i\in \{0,1\}.
\end{equation*}
 Consequently, if ${\bf Y}\in {\bf D_q}(\cH)^-$,  a limiting process implies that
 ${\bf \Delta_{q,Y}^p}(I)\geq 0$ for any ${\bf p}=(p_1,\ldots, p_k)$ with  $ p_i\in \{0,1\}$.
 Therefore, ${\bf D_q}(\cH)^-\subseteq{\bf D_q^-}(\cH)$. To prove the reverse inequality, let
${\bf Y}=(Y_1,\ldots, Y_k)\in {\bf D_q^-}(\cH)$. For any $r\in [0,1)$, we have
$$
\Phi_{q_i,r{\bf W}_i}(I)=\sum_{\alpha\in \FF_{n_i}^+, |\alpha|\geq 1} a_\alpha r^{2|\alpha|}{\bf W}_{i,\alpha}{\bf W}_{i,\alpha}^*\leq r^2\Phi_{q_i,{\bf W}_i}(I)\leq r^2I.
$$
Hence, $\prod_{i=1}^k(I-\Phi_{q_i,rW_i}(I))\geq (1-r^2)^kI$.
  Applying the Berezin transform at ${\bf Y}$ and using the fact that ${\boldsymbol\cB_{\bf X}}$ is a completely positive linear map, we deduce that  ${\bf \Delta}_{q,r{\bf Y}}(I)\geq (1-r^2)^kI$, which shows that $r{\bf Y}\in {\bf D_q}(\cH)$. Since $ r{\bf Y}\to {\bf Y}$, as $r\to 1$, we conclude that ${\bf D_q^-}(\cH)\subseteq  {\bf D_q}(\cH)^-$, which completes the proof of  item (i).

 Now, we prove item (ii). Using the inequality $I-\Phi_{q_i,{\bf z}_i{\bf W}_i}(I)\geq I-\Phi_{q_i,{\bf W}_i}(I)\geq 0$ and the fact that $I-\Phi_{q_i,{\bf W}_i}(I)$ commutes with $I-\Phi_{q_s,{\bf W}_s}(I)$, one can deduce that  if    ${\bf z}=({\bf z}_1,\ldots, {\bf z}_k)$,  where ${\bf z}_i=(z_{i,1},\ldots, z_{i,n_i})\in \overline{\DD}^{n_i}$, then
$$
(id-\Phi_{q_1,{\bf z}_1 {\bf W}_1})^{p_1}\circ\cdots \circ (id-\Phi_{q_k,{\bf z}_k {\bf W}_k})^{p_k}(I)
\geq
(id-\Phi_{q_1.{\bf W}_1})^{p_1}\circ\cdots \circ (id-\Phi_{q_k,{\bf W}_k})^{p_k}(I),\qquad p_i\in \{0,1\}.
$$
If   ${\bf X}\in {\bf D_q}(\cH)$, then applying the Berezin transform at ${\bf X}$ to the inequality above,  we obtain
$
{\bf \Delta_{q, zX}^p}(I)\geq {\bf \Delta_{q,X}^p}(I)>0
$
for any ${\bf p}=(p_1,\ldots, p_k)$ with $ p_i\in \{0,1\}$. This implies
$${\bf z}{\bf D_q}(\cH)\subseteq {\bf D_q}(\cH), \qquad {\bf z}\in  \overline{\DD}^{n_1+\cdots+n_k},
$$
which shows that ${\bf D_q}(\cH)$ is a complete Reinhardt domain and    ${\bf D_q}(\cH)=\bigcup_{{\bf z}\in  \overline{\DD}^{n_1+\cdots+n_k}}{\bf z}{\bf D_q}(\cH)$.

 Now, fix ${\bf X}\in {\bf D_q}(\cH)^-$ and ${\bf z}\in  {\DD}^{n_1+\cdots+n_k}$. Then there is $r\in(0,1)$ such that $\frac{1}{r}{\bf z}\in \DD^{n_1+\cdots +n_k}$.
Applying the Berezin transform at ${\bf X}$ to the  inequality \eqref{qq} when $\lambda_1=\cdots =\lambda_k=r$,  one can see  that $r{\bf X}\in {\bf D_q}(\cH)$.
 Therefore, ${\bf zX}\in\frac{1}{r}{\bf z} {\bf D_q}(\cH)\in {\bf D_q}(\cH)$, which shows that
\begin{equation}
\label{subset1}
{\bf z}{\bf D_q}(\cH)^-\subseteq {\bf D_q}(\cH), \qquad {\bf z}\in  {\DD}^{n_1+\cdots+n_k}.
\end{equation}

Since ${\bf D_q}(\cH)$ is an open set , for any ${\bf X}\in {\bf D_q}(\cH)$, there is $r\in (0,1)$ such that ${\bf X}\in r{\bf D_q}(\cH))$. Consequently,
\begin{equation}
\label{subset2}
{\bf D_q}(\cH)\subset \bigcup_{0\leq r<1}r{\bf D_q}(\cH)\subset
\bigcup_{{\bf z}\in  {\DD}^{n_1+\cdots+n_k}}{\bf z}{\bf D_q}(\cH)\subseteq\bigcup_{{\bf z}\in  {\DD}^{n_1+\cdots+n_k}}{\bf z}{\bf D_q}(\cH)^-
\end{equation}
and
\begin{equation}
\label{subset3}
{\bf D_q}(\cH)\subset \bigcup_{0\leq r<1}r{\bf D_q}(\cH)\subset
\bigcup_{0\leq r<1}r{\bf D_q}(\cH)^-.
 \end{equation}
Using  relations \eqref{subset1} and \eqref{subset2}, one can see  that the first sequence of equalities in item (ii) holds.
 Due to relation \eqref{subset1}, for each $r\in [0,1)$, we have
$ r{\bf D_q}(\cH)^-\subseteq {\bf D_q}(\cH)$ which together with  relation \eqref{subset3} show that the second sequence of equalities in item (ii) holds.   Item (iii) follows easily  from item (ii).
The proof is complete.
\end{proof}

We remark that if  ${\bf r }:=(r_1,\ldots, r_k)$, $r_i>0$, then we also have
 $ {\bf D_q}(\cH)=\bigcup_{0\leq r_i<1}{\bf r}{\bf D_q}(\cH)^-.
$

For each $i\in\{1,\ldots, k\}$, let $Z_i:=(Z_{i,1},\ldots, Z_{i,n_i})$ be
an  $n_i$-tuple of noncommuting indeterminates and assume that, for any
$p,q\in \{1,\ldots, k\}$, $p\neq q$, the entries in $Z_p$ are commuting
 with the entries in $Z_q$. We set $Z_{i,\alpha_i}:=Z_{i,j_1}\cdots Z_{i,j_p}$
  if $\alpha_i\in \FF_{n_i}^+$ and $\alpha_i=g_{j_1}^i\cdots g_{j_p}^i$, and
   $Z_{i,g_0^i}:=1$, where $g_0^i$ is the identity in $\FF_{n^i}^+$.
   If  $\boldsymbol\alpha:=(\alpha_1,\ldots, \alpha_k)\in \FF_{n_1}^+\times \cdots \times\FF_{n_k}^+$, we denote ${\bf Z}_{\boldsymbol\alpha}:= {Z}_{1,\alpha_1}\cdots {Z}_{k,\alpha_k}$.
Let $\ZZ$ be the set of all integers and $\ZZ_+$ be the set of all nonnegative integers.

 If $T_1,\ldots, T_n \in B(\cH)$, we use the notation $[T_1,\ldots, T_n]$ to denote either the $n$-tuple $(T_1,\ldots, T_n)\in B(\cH)^n$ or the row operator $[T_1\, \cdots \, T_n]$ acting from the direct sum $\cH^{(n)}:=\cH\oplus \cdots \oplus \cH$ to $\cH$. We also set
$\Lambda_{\bf p}:=\{\boldsymbol\alpha:=(\alpha_1,\ldots, \alpha_k)\in \FF_{n_1}^+\times\cdots \times \FF_{n_k}^+: \ |\alpha_i|=p_i \}$, where ${\bf p}:=(p_1,\ldots, p_k)\in \ZZ_+^k$.

\begin{lemma}\label{Le} Let ${\bf W}:=({\bf W}_1,\ldots, {\bf W}_k)$  be the  universal model  and  $\{b_{i,\alpha_i}\}$ be the coefficients associated
  with the abstract noncommutative
  polydomain ${\bf D_q^-}$.
   If $A_{(\boldsymbol\alpha)}$, $\boldsymbol\alpha\in  \Lambda_{\bf p}$, are bounded linear operators on a Hilbert space $\cK$, then
 $$\left\|\sum\limits_{\boldsymbol\alpha=(\alpha_1,\ldots, \alpha_k)\in  \Lambda_{\bf p}}{{b_{1,\alpha_1}\cdots b_{k,\alpha_k}}}{\bf W}_{\boldsymbol\alpha}{\bf W}_{\boldsymbol\alpha}^*\right\|=1
$$
and
\begin{equation*}
 \begin{split}
 \left\|\sum_{\boldsymbol\alpha=(\alpha_1,\ldots, \alpha_k)\in  \Lambda_{\bf p}} \frac{1}{{b_{1,\alpha_1}\cdots b_{k,\alpha_k}}}A_{(\boldsymbol\alpha)}^*A_{(\boldsymbol\alpha)}
\right\|^{1/2}&=\left\| \sum_{\boldsymbol\alpha\in  \Lambda_{\bf p}} A_{(\boldsymbol\alpha)} \otimes  {\bf W}_{\boldsymbol\alpha}\right\|=\left\|\sum\limits_{\boldsymbol\alpha\in  \Lambda_{\bf p}} A_{(\boldsymbol\alpha)}^*A_{(\boldsymbol\alpha)}\otimes {\bf W}_{\boldsymbol\alpha}^*{\bf W}_{\boldsymbol\alpha}\right\|^{1/2}.
 \end{split}
 \end{equation*}
 \end{lemma}
\begin{proof}  First note that if $E_1,\ldots, E_m$ are operators on a Hilbert space and have orthogonal ranges, then
$\|[E_1,\ldots, E_m]\|=\max_{j\in \{1,\ldots, m\}} \|E_j\|$. We know from \cite{Po-domains}  that for each $i\in \{1,\ldots, k\}$ and any $\alpha_i\in \FF_{n_i}^+$, $\|W_{i,\alpha_i}\|=\frac{1}{\sqrt{b_{i,\alpha_i}}}$. Hence, if ${\boldsymbol \alpha}=(\alpha_1,\ldots,\alpha_k)\in  \FF_{n_1}^+\times \cdots \times\FF_{n_k}^+$, then
$$
\|{\bf W}_{\boldsymbol \alpha}\|=\frac{1}{\sqrt{b_{1,\alpha_1}\cdots b_{k,\alpha_k}}}.
$$
Since the operators ${\bf W}_{\boldsymbol \alpha}$, $\boldsymbol\alpha=(\alpha_1,\ldots, \alpha_k)\in  \Lambda_{\bf p}$, have orthogonal ranges, we deduce that
$$
\left\|[\sqrt{{{b_{1,\alpha_1}\cdots b_{k,\alpha_k}}}}{\bf W}_{\boldsymbol\alpha}:\ \boldsymbol\alpha=(\alpha_1,\ldots, \alpha_k)\in  \Lambda_{\bf p}]\right\|=1,
$$
which proves the first relation of the lemma.
Consequently, using the fact that $${\bf W}_{\boldsymbol \alpha}(1)=\frac{1}{\sqrt{b_{1,\alpha_1}\cdots b_{k,\alpha_k}}}e^1_{\alpha_1}\otimes\cdots \otimes e_{\alpha_k}^k,
$$
we deduce that, for any $h\in \cH$ with $\|h\|\leq 1$,
\begin{equation*}
\begin{split}
&\left|\sum_{\boldsymbol\alpha=(\alpha_1,\ldots, \alpha_k)\in  \Lambda_{\bf p}} \frac{1}{{b_{1,\alpha_1}\cdots b_{k,\alpha_k}}}\left<A_{(\boldsymbol\alpha)}^*A_{(\boldsymbol\alpha)}h,h\right>
\right|^{1/2}\\
&=
\left\| \sum_{\boldsymbol\alpha\in  \Lambda_{\bf p}}( A_{(\boldsymbol\alpha)} \otimes  {\bf W}_{\boldsymbol\alpha})(h\otimes 1)\right\|\\
&\leq \left\| \sum_{\boldsymbol\alpha\in  \Lambda_{\bf p}} A_{(\boldsymbol\alpha)} \otimes  {\bf W}_{\boldsymbol\alpha}\right\|\\
&\leq
 \left\|\sum_{\boldsymbol\alpha=(\alpha_1,\ldots, \alpha_k)\in  \Lambda_{\bf p}} \frac{1}{{b_{1,\alpha_1}\cdots b_{k,\alpha_k}}}A_{(\boldsymbol\alpha)}^*A_{(\boldsymbol\alpha)}
\right\|^{1/2}\left\|\sum\limits_{\boldsymbol\alpha=(\alpha_1,\ldots, \alpha_k)\in  \Lambda_{\bf p}}{{b_{1,\alpha_1}\cdots b_{k,\alpha_k}}}{\bf W}_{\boldsymbol\alpha}{\bf W}_{\boldsymbol\alpha}^*\right\|^{1/2}\\
&=\left\|\sum_{\boldsymbol\alpha=(\alpha_1,\ldots, \alpha_k)\in  \Lambda_{\bf p}} \frac{1}{{b_{1,\alpha_1}\cdots b_{k,\alpha_k}}}A_{(\boldsymbol\alpha)}^*A_{(\boldsymbol\alpha)}
\right\|^{1/2}.
\end{split}
\end{equation*}
Hence, we deduce that
\begin{equation*}
 \begin{split}
 \left\|\sum_{\boldsymbol\alpha=(\alpha_1,\ldots, \alpha_k)\in  \Lambda_{\bf p}} \frac{1}{{b_{1,\alpha_1}\cdots b_{k,\alpha_k}}}A_{(\boldsymbol\alpha)}^*A_{(\boldsymbol\alpha)}
\right\|^{1/2}&=\left\| \sum_{\boldsymbol\alpha\in  \Lambda_{\bf p}} A_{(\boldsymbol\alpha)} \otimes  {\bf W}_{\boldsymbol\alpha}\right\|.
 \end{split}
 \end{equation*}
 Since ${\bf W}_{\boldsymbol \alpha}$, $\boldsymbol\alpha=(\alpha_1,\ldots, \alpha_k)\in  \Lambda_{\bf p}$, have orthogonal ranges, we deduce that
\begin{equation*}
 \begin{split}
  \left\| \sum_{\boldsymbol\alpha\in  \Lambda_{\bf p}} A_{(\boldsymbol\alpha)} \otimes  {\bf W}_{\boldsymbol\alpha}\right\|=\left\|\sum\limits_{\boldsymbol\alpha\in  \Lambda_{\bf p}} A_{(\boldsymbol\alpha)}^*A_{(\boldsymbol\alpha)}\otimes {\bf W}_{\boldsymbol\alpha}^*{\bf W}_{\boldsymbol\alpha}\right\|^{1/2}.
 \end{split}
 \end{equation*}
 The proof is complete.
\end{proof}

The next result is an analogue of Abel theorem from complex analysis in  our noncommutative  multivariable setting.

\begin{theorem} \label{Abel} Let $\varphi:=\sum\limits_{\boldsymbol\alpha\in \FF_{n_1}^+\times \cdots \times\FF_{n_k}^+} A_{(\boldsymbol\alpha)}\otimes {\bf Z}_{\boldsymbol\alpha}$ be a formal power series with $A_{(\boldsymbol\alpha)}\in B(\cK)$ and  let ${\bf r}=(r_1,\ldots, r_k)$ be such that  $r_i>0$. Then the following statements hold.
\begin{enumerate}
\item[(i)] If the set
$$
\cA:=\left\{ \left\|r_1^{2p_1}\cdots r_k^{2p_k} \sum_{\boldsymbol\alpha\in  \Lambda_{\bf p}} \frac{1}{{b_{1,\alpha_1}\cdots b_{k,\alpha_k}}}A_{(\boldsymbol\alpha)}^*A_{(\boldsymbol\alpha)}
\right\|:\ {\bf p}=(p_1,\ldots, p_k)\in \ZZ_+^k\right\}
$$
is bounded, then the series
 $$
\sum_{{\bf p}\in \ZZ_+^k}\left\| \sum_{\boldsymbol\alpha\in  \Lambda_{\bf p}} A_{(\boldsymbol\alpha)} \otimes  {\bf X}_{\boldsymbol\alpha}\right\|
$$
is convergent in ${\bf r} {\bf D_q}(\cH)$,  the regular polydomain of  polyradius ${\bf r}=(r_1,\ldots, r_k)$,  and uniformly convergent on ${\bf s} {\bf D_q}(\cH)^-$ for any ${\bf s}=(s_1,\ldots, s_k)$ with $0\leq s_i<r_i$.
\item[(ii)] If the set $\cA$ is unbounded, then the series
$$\sum_{{\bf p}\in \ZZ_+^k}\left\| \sum_{\boldsymbol\alpha\in  \Lambda_{\bf p}} A_{(\boldsymbol\alpha)} \otimes  {\bf X}_{\boldsymbol\alpha}\right\|\quad \text{ and } \quad
\sum_{{\bf p}\in \ZZ_+^k} \sum_{\boldsymbol\alpha\in  \Lambda_{\bf p}} A_{(\boldsymbol\alpha)} \otimes  {\bf X}_{\boldsymbol\alpha}
$$
are divergent for some ${\bf X}\in {\bf r}{\bf D_q}(\cH)^-$ and some Hilbert space  $\cH$.
\end{enumerate}
\end{theorem}

\begin{proof} Let $s_i<r_i$ for  $i\in \{1,\ldots, k\}$, and ${\bf X}\in {\bf r} {\bf D_q}(\cH)$. Assume that  there is $C>0$ such that
$$\left\|r_1^{2p_1}\cdots r_k^{2p_k}\sum_{\boldsymbol\alpha=(\alpha_1,\ldots, \alpha_k)\in  \Lambda_{\bf p}} \frac{1}{{b_{1,\alpha_1}\cdots b_{k,\alpha_k}}}A_{(\boldsymbol\alpha)}^*A_{(\boldsymbol\alpha)}\right\|\leq C,\qquad {\bf p}=(p_1,\ldots, p_k)\in \ZZ_+^k.
$$
Due to the noncommutative von Neumann    inequality \cite{Po-Berezin-poly} (see also \cite{Po-poisson}), we have
\begin{equation*}
\begin{split}
\left\|\sum\limits_{\boldsymbol\alpha\in  \Lambda_{\bf p}}A_{(\boldsymbol\alpha)}\otimes{\bf X}_{\boldsymbol\alpha}\right\|
&\leq
\left\|\sum\limits_{\boldsymbol\alpha\in  \Lambda_{\bf p}}s_1^{p_1}\cdots s_k^{p_k}A_{(\boldsymbol\alpha)}\otimes {\bf W}_{\boldsymbol\alpha}\right\|\\
&=s_1^{p_1}\cdots s_k^{p_k}\left\|\sum_{\boldsymbol\alpha=(\alpha_1,\ldots, \alpha_k)\in  \Lambda_{\bf p}} \frac{1}{{b_{1,\alpha_1}\cdots b_{k,\alpha_k}}}A_{(\boldsymbol\alpha)}^*A_{(\boldsymbol\alpha)} \right\|^{1/2}\\
&< \left(\frac{s_1}{r_1}\right)^{p_1}\cdots \left(\frac{s_k}{r_k}\right)^{p_k}C^{1/2}
\end{split}
\end{equation*}
for any ${\bf X}\in s{\bf D_q}(\cH)^-$. On the other hand, due to Proposition \ref{reg-poly}, we have
${\bf r}{\bf D_q}(\cH)= \bigcup_{0\leq s_i<r_i}{\bf s}{\bf D_q}(\cH)^-$.
Now, since the series $\sum_{(p_1,\ldots, p_k)\in \ZZ_+^k} \left(\frac{s_1}{r_1}\right)^{p_1}\cdots \left(\frac{s_k}{r_k}\right)^{p_k}$ is convergent,  one can easily complete the proof of part (i).

To prove part (ii), assume that the set $\cA$ is unbounded. We already know that the tuple  ${\bf rW}:=(r_1{\bf W}_1,\ldots, r_k {\bf W}_k)$ is in the polydomain $  {\bf r D_q}(\otimes_{i=1}^k F^2(H_{n_i}))^-$.
Due to Lemma \ref{Le}, we have
\begin{equation*}
\begin{split}
 \left\|\sum\limits_{\boldsymbol\alpha\in  \Lambda_{\bf p}}A_{(\boldsymbol\alpha)}\otimes r_1^{p_1}\cdots r_k^{p_k}{\bf W}_{\boldsymbol\alpha}\right\|
  =
r_1^{p_1}\cdots r_k^{p_k}\left\|\sum_{\boldsymbol\alpha=(\alpha_1,\ldots, \alpha_k)\in  \Lambda_{\bf p}} \frac{1}{{b_{1,\alpha_1}\cdots b_{k,\alpha_k}}}A_{(\boldsymbol\alpha)}^*A_{(\boldsymbol\alpha)} \right\|^{1/2}.
\end{split}
\end{equation*}
    If  we assume that the series
 $$\sum\limits_{{\bf p}=(p_1,\ldots, p_k)\in \ZZ_+^k}\sum\limits_{\boldsymbol\alpha\in  \Lambda_{\bf p}}A_{(\boldsymbol\alpha)}\otimes r_1^{p_1}\cdots r_k^{p_k}{\bf W}_{\boldsymbol\alpha}
$$
is convergent in the operator norm, then $\left\{\left\|\sum\limits_{\boldsymbol\alpha\in  \Lambda_{\bf p}}A_{(\boldsymbol\alpha)}\otimes r_1^{p_1}\cdots r_k^{p_k}{\bf W}_{\boldsymbol\alpha}\right\|\right\}_{{\bf p}:=(p_1,\ldots, p_k)\in \ZZ_+^k}$  is a bounded sequence,
which contradicts that $\cA$ is an unbounded set. The proof is complete.
\end{proof}

 \begin{definition} A formal power series  $\varphi:=\sum\limits_{\boldsymbol\alpha\in \FF_{n_1}^+\times \cdots \times\FF_{n_k}^+} A_{(\boldsymbol\alpha)}\otimes {\bf Z}_{\boldsymbol\alpha}$ is called   {\it free holomorphic function} (with coefficients in $B(\cK)$) on the
{\it abstract  polydomain}
$\boldsymbol\rho{\bf D_q}:=\coprod_{\cH}\boldsymbol\rho{\bf D_q}(\cH)$, $\boldsymbol\rho=(\rho_1,\ldots,\rho_k)$,  $\rho_i>0$, if the series
$$
\varphi({\bf X} ):=\sum\limits_{{\bf p}\in \ZZ_+^k}\sum\limits_{\boldsymbol\alpha\in  \Lambda_{\bf p}} A_{(\boldsymbol\alpha)}\otimes  {\bf X}_{\boldsymbol\alpha}
$$
is convergent in the operator norm topology for any ${\bf X}=\{X_{i,j}\}\in \boldsymbol\rho{\bf D_q}(\cH)$    and any Hilbert space $\cH$. We denote by $Hol({\boldsymbol\rho {\bf D_q}})$ the set of all free holomorphic functions on  ${\bf \boldsymbol\rho {\bf D_q}}$ with scalar coefficients.
\end{definition}

Using Theorem \ref{Abel}, one can easily deduce the following characterization for free holomorphic functions on regular polydomains.

\begin{corollary} \label{Cara} Let   ${\bf W} $
    be the universal model associated with the abstract regular polydomain
  ${\bf D_q}$. A formal power series $\varphi=\sum\limits_{\boldsymbol\alpha\in \FF_{n_1}^+\times \cdots \times\FF_{n_k}^+} A_{(\boldsymbol\alpha)}\otimes {\bf Z}_{\boldsymbol\alpha}$  is a  free holomorphic function (with coefficients in $B(\cK)$) on the
abstract  polydomain
$\boldsymbol\rho{\bf D_q}$, where
  $\boldsymbol\rho=(\rho_1,\ldots,\rho_k)$,  $\rho_i>0$, if and only if  the series
$$\sum\limits_{{\bf p}=(p_1,\ldots, p_k)\in \ZZ_+^k}\left\|\sum\limits_{\boldsymbol\alpha\in  \Lambda_{\bf p}}A_{(\boldsymbol\alpha)}\otimes r_1^{p_1}\cdots r_k^{p_k}{\bf W}_{\boldsymbol\alpha}\right\|
$$
converges for any $r_i\in [0, \rho_i)$ and  $i\in \{1,\ldots, k\}$.
\end{corollary}
Throughout the paper, we say that the abstract polydomain ${\bf D_q}$ or a free holomorphic function $F$ on ${\bf D_q}$ has a certain property, if the property holds for any Hilbert space representation of ${\bf D_q}$ and $F$, respectively.
We remark that the coefficients of a free holomorphic function on a polydomain  ${\bf D_q}$ are uniquely determined by its representation on an infinite dimensional Hilbert space.

\begin{corollary} \label{scalar} If $\varphi:=\sum\limits_{\boldsymbol\alpha\in \FF_{n_1}^+\times \cdots \times\FF_{n_k}^+} a_{(\boldsymbol\alpha)} {\bf Z}_{\boldsymbol\alpha}$, $a_{(\alpha)}\in \CC$,  is a  free holomorphic function  on the
abstract  polydomain
$\boldsymbol\rho{\bf D_q}$,  $\rho=(\rho_1,\ldots, \rho_k)$, then its representation on $\CC$, i.e.
$$
\varphi(\lambda_1,\ldots, \lambda_k):=\sum\limits_{\boldsymbol\alpha\in \FF_{n_1}^+\times \cdots \times\FF_{n_k}^+} a_{(\boldsymbol\alpha)} \lambda_{\boldsymbol\alpha},\quad \lambda_i=(\lambda_{i,1},\ldots, \lambda_{i, n_i}),
$$
is a holomorphic function on the scalar polydomain $\boldsymbol\rho{\bf D_q}(\CC)$.
\end{corollary}

In what follows, we obtain Cauchy type inequalities for the coefficients of  free holomorphic functions on regular polydomains.
\begin{theorem}\label{Cauchy-ineq} Let $F $ be a free holomorphic function on the polydomain $\boldsymbol \rho{\bf D_q}$, with representation
$$F({\bf X}):=\sum\limits_{(p_1,\ldots, p_k)\in \ZZ_+^k}\sum\limits_{\boldsymbol\alpha\in  \Lambda_{\bf p}}A_{(\boldsymbol\alpha)}\otimes{\bf  X}_{\boldsymbol\alpha}, \qquad {\bf X}\in \boldsymbol \rho{\bf D_q}(\cH),
$$
where $A_{(\boldsymbol\alpha)}\in B(\cK)$.
Let ${\bf r}:=(r_1,\ldots, r_k)$ be such that $0<r_i<\rho_i$ and define
$$M({\bf r}):=\sup\|F({\bf X})\|,
$$
where the supremum is taken over all ${{\bf X}\in {\bf r}{\bf D_q}(\cH)^-}$ and any Hilbert space $\cH$.
Then, for each $k$-tuple ${\bf p}:=(p_1,\ldots, p_k)\in \ZZ_+^k$, we have
$$
\left\|\sum_{\boldsymbol\alpha=(\alpha_1,\ldots, \alpha_k)\in  \Lambda_{\bf p}} \frac{1}{{b_{1,\alpha_1}\cdots b_{k,\alpha_k}}}A_{(\boldsymbol\alpha)}^*A_{(\boldsymbol\alpha)}\right\|^{1/2}\leq \frac{1}{r_1^{p_1}\cdots r_k^{p_k}}M({\bf r}).
$$
Moreover, $M({\bf r})=\|F({\bf rW})\|$, where ${\bf W}$ is the universal model of the regular polydomain ${\bf D_q}$.
\end{theorem}
\begin{proof}
Using the fact that the operators ${\bf W}_{\boldsymbol\alpha}$, with
 $\boldsymbol\alpha=(\alpha_1,\ldots, \alpha_k)\in \FF_{n_1}^+\times\cdots\times \FF_{n_k}^+$, $|\alpha_i|=p_i$, have orthogonal ranges, and Lemma \ref{Le}, we deduce that
\begin{equation*}
\begin{split}
&\left|\left< \left(\sum\limits_{\boldsymbol\alpha\in  \Lambda_{\bf p}}A_{(\boldsymbol\alpha)}^*\otimes  {\bf W}_{\boldsymbol\alpha}^*\right)F({\bf r W})(h\otimes 1), h\otimes 1\right>\right|\\
&\qquad \leq \left\|\sum\limits_{\boldsymbol\alpha\in  \Lambda_{\bf p}}A_{(\boldsymbol\alpha)}^*\otimes  {\bf W}_{\boldsymbol\alpha}^*\right\| M({\bf r}) \|h\|^2\\
&\qquad =
\left\|\sum\limits_{\boldsymbol\alpha\in  \Lambda_{\bf p}}\frac{1}{{b_{1,\alpha_1}\cdots b_{k,\alpha_k}}}A_{(\boldsymbol\alpha)}^*A_{(\boldsymbol\alpha)}\right\|^{1/2}
\left\|\sum\limits_{\boldsymbol\alpha\in  \Lambda_{\bf p}}{{b_{1,\alpha_1}\cdots b_{k,\alpha_k}}}W_{\boldsymbol\alpha}W_{\boldsymbol\alpha}^*\right\|^{1/2}
M({\bf r}) \|h\|^2\\
&\qquad
=\left\|\sum\limits_{\boldsymbol\alpha\in  \Lambda_{\bf p}}\frac{1}{{b_{1,\alpha_1}\cdots b_{k,\alpha_k}}}A_{(\boldsymbol\alpha)}^*A_{(\boldsymbol\alpha)}\right\|^{1/2}M({\bf r}) \|h\|^2
\end{split}
\end{equation*}
for any $h\in \cK$. On the other hand, we have
\begin{equation*}
\begin{split}
&\left< \left(\sum\limits_{\boldsymbol\alpha\in  \Lambda_{\bf p}}A_{(\boldsymbol\alpha)}^*\otimes  {\bf W}_{\boldsymbol\alpha}^*\right)F({\bf r W})(h\otimes 1), h\otimes 1\right>\\
&=r_1^{p_1}\cdots r_k^{p_k}
\left< \left(\sum\limits_{\boldsymbol\alpha\in  \Lambda_{\bf p}}A_{(\boldsymbol\alpha)}^* { A}_{(\boldsymbol\alpha)}\otimes {\bf W}_{\boldsymbol\alpha}^*{\bf W}_{\boldsymbol\alpha}\right) (h\otimes 1), h\otimes 1\right>\\
&=r_1^{p_1}\cdots r_k^{p_k}\left\|\left(\sum\limits_{\boldsymbol\alpha\in  \Lambda_{\bf p}}\frac{1}{{b_{1,\alpha_1}\cdots b_{k,\alpha_k}}}A_{(\boldsymbol\alpha)}^*A_{(\boldsymbol\alpha)}\right)^{1/2}h\right\|^2.
\end{split}
\end{equation*}
Hence, using the   inequality above , we deduce that

$$
r_1^{p_1}\cdots r_k^{p_k}\left\|\left(\sum\limits_{\boldsymbol\alpha\in  \Lambda_{\bf p}}\frac{1}{{b_{1,\alpha_1}\cdots b_{k,\alpha_k}}}A_{(\boldsymbol\alpha)}^*A_{(\boldsymbol\alpha)}\right)^{1/2}h\right\|^2 \leq
\left\|\sum\limits_{\boldsymbol\alpha\in  \Lambda_{\bf p}}\frac{1}{{b_{1,\alpha_1}\cdots b_{k,\alpha_k}}}A_{(\boldsymbol\alpha)}^*A_{(\boldsymbol\alpha)}\right\|^{1/2}M({\bf r}) \|h\|^2
$$
for any $h\in \cK$. Now, the inequality in the theorem follows.
The fact that $M({\bf r})=\|F({\bf rW})\|$ is due to the noncommutative von Neumann inequality \cite{Po-poisson}. The proof is complete.
\end{proof}

Note  that due to the fact that there is $r\in (0,1)$ such that $r{\bf P_n}(\cH)\subset
{\bf D_q}(\cH)$, we have
$$
 B(\cH)^{n_1}\times_c\cdots \times_c B(\cH)^{n_k}=\bigcup_{\rho>0} \rho{\bf D_q}(\cH).
$$
Assume that $\cH$ is an infinite dimensional separable Hilbert space.
We say that $F$ is an {\it entire function} in $B(\cH)^{n_1}\times_c\cdots \times_c B(\cH)^{n_k}$ if $F$ is free holomorphic on every regular polydomain $\rho{\bf D_q}(\cH)$, $\rho>0$.

In what follows, we obtain   an analogue of Liouville's theorem for entire functions  on $B(\cH)^{n_1}\times_c\cdots \times_c B(\cH)^{n_k}$.

\begin{theorem} If $F:B(\cH)^{n_1}\times_c\cdots \times_c B(\cH)^{n_k}\to B(\cK)\otimes_{min} B(\cH)$ is an entire function  with the property that there is a constant $C>0$  and ${\bf m}:=(m_1,\ldots, m_k)\in \ZZ_+^k$ such that
$$
\|F({\bf X})\|\leq C \left\|\sum\limits_{\boldsymbol\alpha\in  \Lambda_{\bf m}}{{b_{1,\alpha_1}\cdots b_{k,\alpha_k}}}{\bf X}_{\boldsymbol\alpha}{\bf X}_{\boldsymbol\alpha}^*\right\|^{1/2}
$$
for any ${\bf X}\in B(\cH)^{n_1}\times_c\cdots \times_c B(\cH)^{n_k}$, then $F$ is a polynomial of degree at most $m_1+\cdots+m_k$. In particular, a bounded free holomorphic function must be constant.
\end{theorem}
\begin{proof} Let $F$ have the representation
$$F({\bf X})=\sum\limits_{{\bf p}=(p_1,\ldots, p_k)\in \ZZ_+^k}\sum\limits_{\boldsymbol\alpha\in  \Lambda_{\bf p}}A_{(\boldsymbol\alpha)}\otimes {\bf X}_{\boldsymbol\alpha}, \qquad {\bf X}\in B(\cH)^{n_1}\times_c\cdots \times_c B(\cH)^{n_k}.
$$
Due to the hypothesis, we have
$$\|F({\bf rW})\|\leq C r_1^{m_1}\cdots r_k^{m_k}\left\|\sum\limits_{\boldsymbol\alpha\in  \Lambda_{\bf m}}{{b_{1,\alpha_1}\cdots b_{k,\alpha_k}}}{\bf W}_{\boldsymbol\alpha}{\bf W}_{\boldsymbol\alpha}^*\right\|^{1/2}\leq C r_1^{m_1}\cdots r_k^{m_k}
$$
for any $r_i>0$.
Using  Theorem \ref{Cauchy-ineq}, we obtain
\begin{equation*}
\begin{split}
\left\|\sum\limits_{\boldsymbol\alpha\in  \Lambda_{\bf p}}\frac{1}{{b_{1,\alpha_1}\cdots b_{k,\alpha_k}}}A_{(\boldsymbol\alpha)}^*A_{(\boldsymbol\alpha)}\right\|^{1/2}
&\leq \frac{1}{r_1^{p_1}\cdots r_k^{p_k}}M({\bf r})
\leq  \frac{1}{r_1^{p_1}\cdots r_k^{p_k}} \|F({\bf rW})\|\\
&\leq  C\frac{1}{r_1^{p_1-m_1}\cdots r_k^{p_k-m_k}}
\end{split}
\end{equation*}
for any $r_i>0$ and $i\in\{1,\ldots, k\}$.
Note that, if there is $s\in\{1,\ldots, k\}$ such that  $p_s>m_s$, then taking $r_s\to \infty$ we obtain $$\sum\limits_{\boldsymbol\alpha\in  \Lambda_{\bf p}}\frac{1}{{b_{1,\alpha_1}\cdots b_{k,\alpha_k}}}A_{(\boldsymbol\alpha)}^*A_{(\boldsymbol\alpha)}=0,
$$
which implies $A_{(\boldsymbol\alpha)}=0$ for any $\boldsymbol\alpha=(\alpha_1,\ldots, \alpha_k)$ with $\alpha_i\in \FF_{n_i}^+$ and  $|\alpha_i|=p_i$ and any $p_i\in \ZZ^+$, $i\neq s$. Hence, we deduce that
$$F({\bf X})=\sum\limits_{{(p_1,\ldots, p_k)\in \ZZ_+^k}\atop{p_i\leq m_i}}\sum\limits_{\boldsymbol\alpha\in  \Lambda_{\bf p}}A_{(\boldsymbol\alpha)}\otimes {\bf X}_{\boldsymbol\alpha}.
$$
The proof is complete.
\end{proof}

Define the set
$$
\Omega:=\left\{{\bf r}:=(r_1,\ldots, r_k)\in \RR_+^k: \ \left\{\left\|r_1^{2p_1}\cdots r_k^{2p_k} \sum_{\boldsymbol\alpha\in  \Lambda_{\bf p}} \frac{1}{{b_{1,\alpha_1}\cdots b_{k,\alpha_k}}}A_{(\boldsymbol\alpha)}^*A_{(\boldsymbol\alpha)}
\right\|\right\}_{{\bf p}:=(p_1,\ldots, p_k)\in \ZZ_+^k} \ \text{ is bounded}\right\}.
$$
  Given a  formal power series $\psi:=\sum\limits_{\boldsymbol\alpha\in \FF_{n_1}^+\times \cdots \times\FF_{n_k}^+} A_{(\boldsymbol\alpha)}\otimes {\bf Z}_{\boldsymbol\alpha}$ , we define
 $$
 {\bf C}_\psi(\cH):=\bigcup_{{\bf r}\in \Omega} {\bf r}{\bf D_q}(\cH)\quad \text{ and } \quad {\bf C}_\psi:=\coprod_\cH {\bf C}_\psi(\cH).
 $$
We say that ${\bf C}_\psi$ is {\it logarithmically convex} if  $\Omega$ is log-convex, i.e.
the set $$\{(\log r_1,\ldots, \log r_k): \ (r_1,\ldots, r_k)\in \Omega, r_i>0\}$$ is  convex.

\begin{proposition} Let $\psi:=\sum\limits_{\boldsymbol\alpha\in \FF_{n_1}^+\times \cdots \times\FF_{n_k}^+} A_{(\boldsymbol\alpha)}\otimes {\bf Z}_{\boldsymbol\alpha}$  be a formal power series. The following statements hold.
\begin{enumerate}
\item[(i)] $\psi$ is free holomorphic function on  ${\bf C}_\psi$
and $$\psi({\bf X})=\sum\limits_{{\bf p}:=(p_1,\ldots, p_k)\in \ZZ_+^k}\sum\limits_{\boldsymbol\alpha\in  \Lambda_{\bf p}}A_{(\boldsymbol\alpha)}\otimes {\bf X}_{\boldsymbol\alpha}, \qquad {\bf X}\in {\bf C}_\psi,
$$
where the series is  convergent in the operator norm.
\item[(ii)] ${\bf C}_\psi$  is  a logarithmically convex complete Reinhardt domain.
\end{enumerate}
\end{proposition}
\begin{proof}
Item (i) is due  to Theorem \ref{Abel} and   the uniqueness of the representation for free holomorphic functions on polydomains.  To prove part (ii), note that Proposition \ref{reg-poly} implies that  ${\bf C}_\psi$  is  a  complete Reinhardt domain. It remains to show that ${\bf C}_\psi$ is logarithmically convex. To this end,   let $(r_1,\ldots, r_k)$ and
$(s_1,\ldots, s_k)$ be in $\Omega$. Then there is a constant $C>0$ such that
$$
\|r_1^{2p_1}\cdots r_k^{2p_k}\Gamma_{\bf p} \|
\leq C\quad \text{ and }\quad
\|s_1^{2p_1}\cdots s_k^{2p_k}\Gamma_{\bf p}  \|\leq C
$$
for any ${\bf p}=(p_1,\ldots, p_k)\in \ZZ_+^k$, where
 $\Gamma_{\bf p}:=\sum\limits_{\boldsymbol\alpha\in  \Lambda_{\bf p}}\frac{1}{{b_{1,\alpha_1}\cdots b_{k,\alpha_k}}}A_{(\boldsymbol\alpha)}^*A_{(\boldsymbol\alpha)}$.  Due to the spectral theorem for positive operators, for any $t\in [0,1]$,  we have
\begin{equation*}
\begin{split}
\left\|( r_1^t s_1^{1-t})^{2p_1}\cdots ( r_k^t s_k^{1-t})^{2p_k}\Gamma_{\bf p}\right\|
&=
\left\|(r_1^{2p_1}\cdots r_k^{2p_k} \Gamma_{\bf p})^t (s_1^{2p_1}\cdots s_k^{2p_k} \Gamma_{\bf p})^{1-t}\right\|\\
&\leq
\left\|(r_1^{2p_1}\cdots r_k^{2p_k} \Gamma_{\bf p})^t\right\| \left\|(s_1^{2p_1}\cdots s_k^{2p_k} \Gamma_{\bf p})^{1-t}\right\|\\
&\leq
\left\|r_1^{2p_1}\cdots r_k^{2p_k} \Gamma_{\bf p}\right\|^t \left\|s_1^{2p_1}\cdots s_k^{2p_k} \Gamma_{\bf p}\right\|^{1-t}\\
&\leq C^t C^{1-t}=C.
\end{split}
\end{equation*}
Hence,  $(r_1^t s_1^{1-t},\ldots,  r_k^t k_1^{1-t})\in \Omega$, which proves that
 ${\bf C}_\psi$  is   logarithmically convex.
The proof is complete.
\end{proof}

We remark that, due to Theorem \ref{Abel},  if $\boldsymbol\rho:=(\rho_1,\ldots, \rho_k)\notin \Omega$, then
$\sum\limits_{(p_1,\ldots, p_k)\in \ZZ_+^k}\sum\limits_{\boldsymbol\alpha\in  \Lambda_{\bf p}}A_{(\alpha)}\otimes {\bf X}_{\boldsymbol\alpha}
$
is  divergent for some ${\bf X}\in {\boldsymbol\rho}{\bf D_q}(\cH)^-$ and some Hilbert space  $\cH$.
We call the set ${\bf C}_\psi$  the {\it universal domain of convergence} of the power series $\psi$.
In what follows, we  find the largest polydomain $r{\bf D_q}$, $r>0$, which is included in the universal domain of convergence ${\bf C}_\psi$.

 \begin{theorem} \label{Hadamard}
 Let $\psi:=\sum\limits_{\boldsymbol\alpha\in \FF_{n_1}^+\times \cdots \times\FF_{n_k}^+} A_{(\boldsymbol\alpha)}\otimes {\bf Z}_{\boldsymbol\alpha}$   be a formal power series
 and define
 $\gamma\in[0,\infty]$ by setting
 $$
 \frac{1}{\gamma}:=\limsup_{{\bf p}:=(p_1,\ldots, p_k)\in \ZZ_+^k} \left\|\sum\limits_{\boldsymbol\alpha\in  \Lambda_{\bf p}}\frac{1}{{b_{1,\alpha_1}\cdots b_{k,\alpha_k}}}A_{(\boldsymbol\alpha)}^*A_{(\boldsymbol\alpha)}\right\|^{\frac{1}{2(p_1+\cdots +p_k)}}.
 $$
 Then the following statements hold.
 \begin{enumerate}
 \item[(i)]  The series
 $$
 \sum\limits_{{\bf p}\in \ZZ_+^k}\left\|\sum\limits_{\boldsymbol\alpha\in  \Lambda_{\bf p}}A_{(\alpha)}\otimes {\bf X}_{\boldsymbol\alpha}\right\|, \qquad {\bf X}\in \gamma{\bf D_q}(\cH),
 $$
 is convergent. Moreover, the convergence is uniform on $r{\bf D_q}(\cH)^-$ if $0\leq r<\gamma$.
 \item[(ii)] If $\gamma\in [0,\infty)$ and
   $s>\gamma$, then there is a Hilbert space $\cH$ and ${\bf Y}\in s{\bf D_q}(\cH)^-$  such that the series
 $$
 \sum\limits_{{\bf p}\in \ZZ_+^k}\sum\limits_{\boldsymbol\alpha\in  \Lambda_{\bf p}}A_{(\alpha)}\otimes {\bf Y}_{\boldsymbol\alpha}
 $$
 is divergent in the operator norm topology.
 \end{enumerate}
 \end{theorem}
 \begin{proof} First, we consider the case when  $\gamma\in (0,\infty)$. Let ${\bf X}\in r{\bf D_q}(\cH)^-$ be such that $0\leq r<\gamma$ and let $\rho\in (r,\gamma)$. Then
 \begin{equation*}
 \left\|\sum\limits_{\boldsymbol\alpha\in  \Lambda_{\bf p}}\frac{1}{{b_{1,\alpha_1}\cdots b_{k,\alpha_k}}}A_{(\boldsymbol\alpha)}^*A_{(\boldsymbol\alpha)}\right\|^{\frac{1}{2(p_1+\cdots +p_k)}}<\frac{1}{\rho}
 \end{equation*}
for all but finitely many ${\bf p}:=(p_1,\ldots, p_k)\in \ZZ_+^k$. Consequently, due to the noncommutative  von Neumann   inequality \cite{Po-poisson}, we have
\begin{equation*}
\begin{split}
\left\|\sum\limits_{\boldsymbol\alpha\in  \Lambda_{\bf p}}A_{(\boldsymbol\alpha)}\otimes {\bf X}_{\boldsymbol\alpha}\right\|
&\leq
\left\|\sum\limits_{\boldsymbol\alpha\in  \Lambda_{\bf p}}A_{(\boldsymbol\alpha)}\otimes r^{p_1+\cdots + p_k} {\bf W}_{(\alpha)}\right\|\\
&=r^{p_1+\cdots p_k}\left\|\sum\limits_{\boldsymbol\alpha\in  \Lambda_{\bf p}}A_{(\boldsymbol\alpha)}^*A_{(\boldsymbol\alpha)}\otimes {\bf W}_{\boldsymbol\alpha}^*{\bf W}_{\boldsymbol\alpha}\right\|^{1/2}<\left(\frac{r}{\rho}\right)^{p_1+\cdots +p_k}
\end{split}
\end{equation*}
for all but finitely many ${\bf p}:=(p_1,\ldots, p_k)\in \ZZ_+^k$. As a consequence, item (i) holds and  implies that   the series
 $
 \sum\limits_{{\bf p}\in \ZZ_+^k}^\infty\left\|\sum\limits_{\boldsymbol\alpha\in  \Lambda_{\bf p}}A_{(\boldsymbol\alpha)}\otimes {\bf X}_{\boldsymbol\alpha}\right\|
 $
  is uniformly convergent  on $r{\bf D_q}(\cH)^-$. The case when $\gamma=\infty$ can be treated in a similar manner. We leave it to the reader.
Now, assume that $\gamma\in [0,\infty)$ and $\gamma<\rho<s$. Let ${\bf Y}:=s{\bf W}$, where ${\bf W}$ is the universal model of ${\bf D_q^-}$. As above if  ${\bf Y}\in s{\bf D_q}(\cH)^-$ then
\begin{equation}
\label{AAA}
\left\|\sum\limits_{\boldsymbol\alpha\in  \Lambda_{\bf p}}A_{(\boldsymbol\alpha)}\otimes {\bf Y}_{\boldsymbol\alpha}\right\|=
s^{p_1+\cdots +p_k}\left\|\sum\limits_{\boldsymbol\alpha\in  \Lambda_{\bf p}}A_{(\boldsymbol\alpha)}^*A_{(\boldsymbol\alpha)}\otimes {\bf W}_{\boldsymbol\alpha}^*{\bf W}_{\boldsymbol\alpha}\right\|^{1/2}.
\end{equation}
Since $\frac{1}{\rho}<\frac{1}{\gamma}$, there are infinitely many tuples
${\bf p}:=(p_1,\ldots, p_k)\in \ZZ_+^k$ such that
$$\left\|\sum\limits_{\boldsymbol\alpha\in  \Lambda_{\bf p}}A_{(\boldsymbol\alpha)}^*A_{(\boldsymbol\alpha)}\otimes {\bf W}_{\boldsymbol\alpha}^*{\bf W}_{\boldsymbol\alpha}\right\|^{\frac{1}{2(p_1+\cdots +p_k)}}>\frac{1}{\rho}.
 $$
Hence, and using relation \eqref{AAA},  we deduce that
$\|\sum\limits_{\boldsymbol\alpha\in  \Lambda_{\bf p}}A_{(\boldsymbol\alpha)}\otimes {\bf Y}_{\boldsymbol\alpha}\|>\left(\frac{s}{\rho}\right)^{p_1+\cdots+p_k}$. This shows
that the series
 $$
 \sum\limits_{{\bf p}:=(p_1,\ldots, p_k)\in \ZZ_+^k}\left\|\sum\limits_{\boldsymbol\alpha\in  \Lambda_{\bf p}}A_{(\boldsymbol\alpha)}\otimes {\bf Y}_{\boldsymbol\alpha}\right\|
 $$
 is divergent and also that   item (ii) holds. The proof is complete.
  \end{proof}

 The number $\gamma$ satisfying properties (i) and (ii) in  Theorem \ref{Hadamard} is unique and is called the {\it polydomain radius of convergence} for the power series $\psi$.

 In what follows, we set
 $\Gamma_{m}:=\{\boldsymbol\alpha=(\alpha_1,\ldots, \alpha_k)\in \FF_{n_1}^+\times\cdots \times \FF_{n_k}^+: \  |\alpha_1|+\cdots +|\alpha_k|=m\}$.

 \begin{theorem} \label{Hadamard2} Let $\psi:=\sum\limits_{\boldsymbol\alpha\in \FF_{n_1}^+\times \cdots \times\FF_{n_k}^+} A_{(\boldsymbol\alpha)}\otimes {\bf Z}_{\boldsymbol\alpha}$   be a formal power series
 and let
 $\gamma\in[0,\infty]$ be its polydomain radius of convergence. Then the following statements hold.
  \begin{enumerate}
 \item[(i)]  The series
 $$
 \sum\limits_{m=0}^\infty\left\|\sum\limits_{ \boldsymbol\alpha\in \Gamma_m}A_{(\boldsymbol\alpha)}\otimes {\bf X}_{\boldsymbol\alpha}\right\|
 $$
  is uniformly convergent  on $r{\bf D_q}(\cH)^-$ if $0\leq r<\gamma$.
 \item[(ii)]
 For any $s>\gamma$, there is ${\bf Y}\in s{\bf D_q}(\cH)^-$ such that the series
 $$
 \sum\limits_{m=0}^\infty\sum\limits_{ \boldsymbol\alpha\in \Gamma_m}A_{(\boldsymbol\alpha)}\otimes {\bf X}_{\boldsymbol\alpha}
 $$
 is divergent in the operator norm topology.
 \end{enumerate}
 \end{theorem}

 \begin{proof}  Since
 $$
 \sum\limits_{m=0}^\infty\left\|\sum\limits_{ \boldsymbol\alpha\in \Gamma_m}A_{(\boldsymbol\alpha)}\otimes {\bf X}_{\boldsymbol\alpha}\right\|
 \leq
 \sum_{m=0}^\infty  \sum_{{\bf p}:={(p_1,\ldots, p_k)\in\ZZ_+^k}\atop{p_1+\cdots +p_k=m}}
 \left\|\sum\limits_{\boldsymbol\alpha\in  \Lambda_{\bf p}}A_{(\boldsymbol\alpha)}\otimes {\bf X}_{\boldsymbol\alpha}\right\|,
 $$
Theorem \ref{Hadamard} implies  that  item (i) holds.   To   prove   item (ii) is enough to show that, under the condition $\gamma<\rho<s$,
$$
 \sum\limits_{m=0}^\infty\sum\limits_{ \boldsymbol\alpha\in \Gamma_m}A_{(\boldsymbol\alpha)}\otimes s^m{\bf W}_{\boldsymbol\alpha}
 $$
 is divergent in the operator norm topology. Assume that the series above is convergent  and apply it  to the vector $x\otimes 1$, where  $x\in \cK$. Consequently,
$ \sum\limits_{m=0}^\infty\sum\limits_{ \boldsymbol\alpha\in \Gamma_m}A_{(\boldsymbol\alpha)}x\otimes s^m\frac{1}{\sqrt{b_{1,\alpha_1}\cdots b_{k,\alpha_k}}} e_{\boldsymbol\alpha}
$
is in the Hilbert space $\cK\otimes\bigotimes_{i=1}^k F^2(H_{n_i})$. Since $\{e_{\boldsymbol\alpha}\}_{\boldsymbol\alpha\in \FF_{n_1}^+\times\cdots \times  \FF_{n_k}^+}$ is an orthonormal basis for $\bigotimes_{i=1}^k F^2(H_{n_i})$, we deduce that  the series
$\sum_{\boldsymbol\alpha\in \FF_{n_1}^+\times\cdots \times  \FF_{n_k}^+} s^{2(|\alpha_1|+\cdots +|\alpha_k|)}\frac{1}{{b_{1,\alpha_1}\cdots b_{k,\alpha_k}}}A_{(\boldsymbol\alpha)}^*A_{(\boldsymbol\alpha)}$ is WOT-convergent.
For each  $r\in [0,1)$, Lemma \ref{Le}  implies
\begin{equation*}
 \begin{split}
 &\sum_{m=0}^\infty r^m\sum_{{(p_1,\ldots, p_k)\in\ZZ_+^k}\atop{p_1+\cdots +p_k=m}}
 \left\|\sum\limits_{{ \beta_i\in \FF_{n_i}^+, |\beta_i|=p_i}\atop{i\in \{1,\ldots, k\} }}A_{(\boldsymbol\beta)}\otimes s^{p_1+\cdots +p_k} {\bf W}_{\boldsymbol\beta}\right\| \\
 &\leq
 \sum_{m=0}^\infty r^m\sum_{{(p_1,\ldots, p_k)\in\ZZ_+^k}\atop{p_1+\cdots +p_k=m}}
 \left\|\sum\limits_{{ \beta_i\in \FF_{n_i}^+, |\beta_i|=p_i}\atop{i\in \{1,\ldots, k\} }}s^{2(|\beta_1|+\cdots +|\beta_k|)}\frac{1}{{b_{1,\beta_1}\cdots b_{k,\beta_k}}}A_{(\boldsymbol\beta)}^* A_{(\boldsymbol\beta)} \right\|^{1/2}\\&\qquad \times
 \left\|\sum\limits_{{ \beta_i\in \FF_{n_i}^+, |\beta_i|=p_i}\atop{i\in \{1,\ldots, k\} }}{{b_{1,\beta_1}\cdots b_{k,\beta_k}}}{\bf W}_{\boldsymbol\beta}^* {\bf W}_{\boldsymbol\beta} \right\|^{1/2}\\
 &\leq
 \sum_{m=0}^\infty r^m \left(\begin{matrix} m+k-1\\k-1\end{matrix}\right)
 \left\|\sum\limits_{ \boldsymbol\beta=(\beta_1,\ldots, \beta_k)\in \FF_{n_1}^+\times\cdots \times  \FF_{n_k}^+}s^{2(|\beta_1|+\cdots +|\beta_k|)}\frac{1}{{b_{1,\beta_1}\cdots b_{k,\beta_k}}}A_{(\boldsymbol\beta)}^* A_{(\boldsymbol\beta)} \right\|^{1/2}.
\end{split}
 \end{equation*}
Since the latter series is convergent for any $r\in [0,1)$, we deduce that
$$
\sum_{m=0}^\infty r^m\sum_{{(p_1,\ldots, p_k)\in\ZZ_+^k}\atop{p_1+\cdots +p_k=m}}
 \left\|\sum\limits_{{ \beta_i\in \FF_{n_i}^+, |\beta_i|=p_i}\atop{i\in \{1,\ldots, k\} }}A_{(\boldsymbol\beta)}\otimes s^{p_1+\cdots +p_k} {\bf W}_{\boldsymbol\beta}\right\|<\infty,
 $$
which implies
$$
\sum_{m=0}^\infty  \sum_{{(p_1,\ldots, p_k)\in\ZZ_+^k}\atop{p_1+\cdots +p_k=m}}
 \left\|\sum\limits_{{ \beta_i\in \FF_{n_i}^+, |\beta_i|=p_i}\atop{i\in \{1,\ldots, k\} }}A_{(\boldsymbol\beta)}\otimes   {\bf X}_{\boldsymbol\beta}\right\|<\infty
 $$
for any ${\bf X}\in \rho{\bf D_q}(\cH)^-$, where $\rho\in (\gamma, s)$, which contradicts  Theorem \ref{Hadamard}.
Therefore, item (ii) holds and the proof is complete.
\end{proof}

An interesting consequence of the proofs of  Theorem \ref{Hadamard} and Theorem \ref{Hadamard2} is  the following result.

 \begin{corollary}  The  polydomain radius of convergence of a power series $\varphi:=\sum\limits_{\boldsymbol\alpha\in \FF_{n_1}^+\times \cdots \times\FF_{n_k}^+} A_{(\boldsymbol\alpha)}\otimes {\bf Z}_{\boldsymbol\alpha}$   satisfies the relation
 \begin{equation*}\begin{split}
 \gamma &=\sup\left\{r\geq 0:\
 \sum\limits_{m=0}^\infty\sum\limits_{  {|\alpha_1|+\cdots +|\alpha_k|=m}\atop {\alpha_i\in \FF_{n_i}^+}}A_{(\boldsymbol\alpha)}\otimes r^q{\bf W}_{\boldsymbol\alpha} \ \text{is convergent in the operator norm}\right\}\\
 &=
 \sup\left\{r\geq 0:\
 \sum\limits_{(p_1,\ldots, p_k)\in \ZZ_+^k}\sum\limits_{{ \alpha_i\in \FF_{n_i}^+, |\alpha_i|=p_i}\atop{i\in \{1,\ldots, k\} }}A_{(\alpha)}\otimes  r^{p_1+\cdots +p_k}{\bf W}_{\boldsymbol\alpha} \ \text{is convergent in the operator norm}\right\}.
 \end{split}
 \end{equation*}
 \end{corollary}
We  also have the following characterization for free holomorphic functions on polydomains.

\begin{corollary} Let   ${\bf W} $
    be the universal model associated with the abstract regular polydomain
  $ {\bf D_q}$. A formal power series $\varphi:=\sum\limits_{\boldsymbol\alpha\in \FF_{n_1}^+\times \cdots \times\FF_{n_k}^+} A_{(\boldsymbol\alpha)}\otimes {\bf Z}_{\boldsymbol\alpha}$ is a  free holomorphic function (with coefficients in $B(\cK)$) on the
abstract  polydomain
$\boldsymbol\rho{\bf D_q}$, where
  $\boldsymbol\rho=(\rho_1,\ldots,\rho_k)$,  $\rho_i>0$, if and only if  the series
$$
\sum_{m=0}^\infty \sum_{{\boldsymbol\alpha\in \FF_{n_1}^+\times \cdots \times\FF_{n_k}^+ }\atop {|\alpha_1|+\cdots +|\alpha_k|=m}} A_{(\alpha)} \otimes  s_1^{|\alpha_1|}\cdots s_k^{|\alpha_k|}{\bf W}_{\boldsymbol\alpha}
$$
is convergent in the operator norm topology for any $s_i\in [0,\rho_i)$.
 Moreover, the set $Hol({\bf \boldsymbol\rho D_q})$  of all free holomorphic functions with scalar coefficients on  ${\bf \boldsymbol\rho D_q}$ is an algebra.
\end{corollary}

\bigskip

\section{Maximum principle and Schwarz lemma on noncommutative polydomains}

In this section we  prove a maximum principle  and  Schwarz type lemma for free holomorphic functions on regular polydomains.

Let
$H^\infty({\bf D_q})$  denote the set of  all
elements $\varphi$ in $Hol({\bf D_q})$     such
that
$$\|\varphi\|_\infty:= \sup \|\varphi({\bf X})\|<\infty,
$$
where the supremum is taken over all   $ {\bf X}\in {\bf D_q}(\cH)$ and any Hilbert space
$\cH$. One can  show that $H^\infty({\bf D_q})$ is a Banach algebra under pointwise multiplication and the
norm $\|\cdot \|_\infty$.
For each $p\in \NN$, we define the norms $\|\cdot
\|_p:M_{p\times p}\left(H^\infty({\bf D_q})\right)\to
[0,\infty)$ by setting
$$
\|[\varphi_{st}]_{p\times p}\|_p:= \sup \|[\varphi_{st}({\bf X})]_{p\times p}\|,
$$
where the supremum is taken over all  $ {\bf X}\in {\bf D_q}(\cH)$ and any Hilbert space
$\cH$.
It is easy to see that the norms  $\|\cdot\|_p$, $p\in \NN$,
determine  an operator space structure  on $H^\infty({\bf D_q})$,
 in the sense of Ruan (\cite{Pa-book}, \cite{Pi-book}).

In \cite{Po-Berezin-poly},
we identified the  noncommutative algebra
${\bf F}^\infty({\bf D_q})$ with the  Hardy subalgebra $H^\infty({\bf D_q})$ of   bounded free holomorphic functions  on
${\bf D_q}$ with scalar coefficients.
More precisely, we proved that  the map
$ \Phi:H^\infty({\bf D_q})\to {\bf F}^\infty({\bf D_q}) $ defined by
$
\Phi\left(\varphi\right):=\text{\rm SOT-}\lim_{r\to 1}\varphi(r{\bf W}),
$
is a completely isometric isomorphism of operator algebras, where
$\varphi(r{\bf W}):=\sum_{q=0}^\infty \sum_{{\boldsymbol\alpha\in \FF_{n_1}^+\times \cdots \times\FF_{n_k}^+ }\atop {|\alpha_1|+\cdots +|\alpha_k|=q}} r^q a_{(\boldsymbol\alpha)} {\bf W}_{\boldsymbol\alpha}$ and the convergence of the series is in the operator norm topology.
Moreover, if  $\varphi$
is  a free holomorphic function on the abstract  polydomain ${\bf D_q}$, then the following statements are equivalent:
 \begin{enumerate}
 \item[(i)]$\varphi\in H^\infty({\bf D_q})$;
\item[(ii)] $\sup\limits_{0\leq r<1}\|\varphi(r{\bf W})\|<\infty$;
\item[(iii)]
there exists $\psi\in {\bf F}^\infty ({\bf D_q})$ such that  $\varphi({\bf X})={\boldsymbol\cB}_{\bf X}[\psi]$ for ${\bf X}\in {\bf D_q}(\cH)$, where ${\boldsymbol\cB}_{\bf X}$ is the  noncommutative Berezin
transform  associated with the abstract   polydomain ${\bf D_q}$.
\end{enumerate}
 Moreover, $\psi$ is uniquely determined by $\varphi$, namely,
$\psi=\text{\rm SOT-}\lim_{r\to 1}\varphi(r{\bf W})$
and
\begin{equation*}
\|\psi\|=\sup_{0\leq r<1}\|\varphi(r{\bf W})\|=
\lim_{r\to 1}\|\varphi(r{\bf W})\|=\|\varphi\|_\infty.
\end{equation*}

We  denote by  $A({\bf D_q})$   the set of all  elements $g$
  in $Hol({\bf D_q})$   such that the mapping
$${\bf D_q}(\cH)\ni {\bf X}\mapsto
g({\bf X})\in B(\cH)$$
 has a continuous extension to  $[{\bf D_q}(\cH)]^-$ for any Hilbert space $\cH$. One can show that  $A({\bf D_q})$ is a  Banach algebra under pointwise
multiplication and the norm $\|\cdot \|_\infty$, and it has an operator space structure under the norms $\|\cdot \|_p$, $p\in \NN$.  Moreover, we can
identify the polydomain algebra $\boldsymbol\cA ({\bf D_q})$ with the subalgebra
 $A({\bf D_q})$.  We proved in \cite{Po-Berezin-poly} that  the map
$ \Phi:A({\bf D_q})\to \boldsymbol\cA({\bf D_q}) $
 defined by
$
\Phi\left(g\right):=\lim_{r\to 1}g(r{\bf W})$,  in the  norm topology,
is a completely isometric isomorphism of operator algebras.
Moreover, if  $g$
is  a free holomorphic function on the abstract  polydomain ${\bf D_q}$, then the following statements are equivalent:
 \begin{enumerate}
 \item[(i)]$g\in A({\bf D_q})$;
\item[(ii)] $g(r{\bf W}):=\sum_{m=0}^\infty \sum_{{\boldsymbol\alpha\in \FF_{n_1}^+\times \cdots \times\FF_{n_k}^+ }\atop {|\alpha_1|+\cdots +|\alpha_k|=m}} r^q  a_{(\boldsymbol\alpha)} {\bf W}_{\boldsymbol\alpha}$ is convergent in the  norm topology  as $r\to 1$;
\item[(iii)]
there exists $\varphi\in \boldsymbol\cA({\bf D_q})$ such that $g({\bf X})={\boldsymbol\cB}_{\bf X}[\varphi]$ for ${\bf X}\in {\bf D_q}(\cH)$, where ${\boldsymbol\cB}_{\bf X}$ is the  noncommutative Berezin
transform  associated with the abstract  polydomain ${\bf D_q}$.
\end{enumerate}
Moreover, $\varphi$ is uniquely determined by $g$, namely,
$\varphi=\lim_{r\to 1}g(r{\bf W})$
and
\begin{equation*}
\|\varphi\|=\sup_{0\leq r<1}\|g(r{\bf W})\|=
\lim_{r\to 1}\|g(r{\bf W})\|=\|g\|_\infty.
\end{equation*}

\begin{proposition}\label{pro}
Let  $G=\sum\limits_{\boldsymbol\alpha\in \FF_{n_1}^+\times \cdots \times\FF_{n_k}^+} c_{(\boldsymbol\alpha)} {\bf Z}_{\boldsymbol\alpha}$  be a  free holomorphic function on the
 polydomain
${\bf D_q}$.
\begin{enumerate}
\item[(i)] If $0<r_1<r_2<1$, then $r_1{\bf D_q^-}\subset
r_2{\bf D_q}\subset{\bf D_q}$ and
$$
\|G(r_1{\bf W})\|\leq \|G(r_2{\bf W})\|.
$$
\item[(ii)] If $0<r<1$, then the map \ $G:r{\bf D_q}(\cH)^-\to B(\cH)$
defined by
$$G({\bf X}):=\sum_{m=0}^\infty \sum_{{\boldsymbol\alpha\in \FF_{n_1}^+\times \cdots \times\FF_{n_k}^+ }\atop {|\alpha_1|+\cdots +|\alpha_k|=m}}   c_{(\boldsymbol\alpha)} {\bf X}_{\boldsymbol\alpha},\qquad {\bf X} \in r{\bf D_q}(\cH)^-,
$$
is continuous and \ $\|G({\bf X})\|\leq \|G(r{\bf W})\|$ for any ${\bf X} \in r{\bf D_q}(\cH)^-$.  Moreover, the series defining $G$  converges uniformly on
$r{\bf D_q}(\cH)^-$ in the operator norm topology.
\end{enumerate}
\end{proposition}

\begin{proof} If $0<r_1<r_2<1$, then the inclusions  $r_1{\bf D_q^-}\subset
r_2{\bf D_q}\subset{\bf D_q}$  are due to Proposition \ref{reg-poly}.
Since $\varphi({\bf W}):=\sum_{m=0}^\infty \sum_{{\boldsymbol\alpha\in \FF_{n_1}^+\times \cdots \times\FF_{n_k}^+ }\atop {|\alpha_1|+\cdots +|\alpha_k|=m}}  c_{(\boldsymbol\alpha)}  r_2^{|\boldsymbol\alpha|} {\bf W}_{\boldsymbol\alpha}$ is in $
\boldsymbol\cA({\bf D_q})$, the noncommutative von Neumann inequality implies
$$\|\varphi(r{\bf W})\|\leq
\|\varphi({\bf W})\|\quad \text{ for any } r\in [0,1).
$$
Taking $r:=\frac{r_1}{r_2}$, we obtain $
\|G(r_1{\bf W})\|\leq \|G(r_2{\bf W})\|$.

To prove part (ii), note that $G(r{\bf W})\in \boldsymbol\cA({\bf D_q}))$
and $\frac{1}{r} {\bf X}\in {\bf D_q}(\cH)^-$.
Using  again the noncommutative von Neumann inequality, we obtain
$$
\|G({\bf X})\|=\left\|G\left(r\left(\frac{1}{r}
{\bf X}\right)\right)\right\|\leq \|G(r{\bf W} )\|
$$
and
$$
\sum_{m=0}^\infty \left\|\sum_{{\boldsymbol\alpha\in \FF_{n_1}^+\times \cdots \times\FF_{n_k}^+ }\atop {|\alpha_1|+\cdots +|\alpha_k|=m}}   c_{(\boldsymbol\alpha)} {\bf X}_{\boldsymbol\alpha}\right\|
\leq \sum_{m=0}^\infty \left\|\sum_{{\boldsymbol\alpha\in \FF_{n_1}^+\times \cdots \times\FF_{n_k}^+ }\atop {|\alpha_1|+\cdots +|\alpha_k|=m}}  c_{(\boldsymbol\alpha)}  r^{|\boldsymbol\alpha|} {\bf W}_{\boldsymbol\alpha}\right\|
$$
for any ${\bf X}\in r{\bf D_q}(\cH)^-$.
Now, one can easily complete the  proof.
\end{proof}

In what follows, we prove a  maximum principle for free holomorphic functions on polydomains.

\begin{theorem}\label{max-mod1}
Let $F$ be a free  holomorphic function on ${\bf D_q}$ and let $r\in [0,1)$.  If $\cH$ is an infinite dimensional Hilbert space, then
\begin{equation*}
\begin{split}
\max\{\|F({\bf X})\|:\  {\bf X}\in r{\bf D_q}(\cH)^-\}=
\max\{\|F({\bf X})\|:\  {\bf X}\in \partial \left(r{\bf D_q}(\cH)\right)\}
 =\|F(r{\bf W})\|.
\end{split}
\end{equation*}
If, in addition, $F$ has a continuous extension $\widetilde{F}$ to ${\bf D_q^-}$  in the operator norm, then
\begin{equation*}
\begin{split}
\max\{\|\widetilde{F}({\bf X})\|:\ {\bf X}\in {\bf D_q}(\cH)^-\}
=\max\{\|\widetilde{F}({\bf X})\|:\ {\bf X}\in \partial{\bf D_q}(\cH)\} = \|F \|_\infty.
\end{split}
\end{equation*}
\end{theorem}

\begin{proof} Due to the noncommutative von Neumann inequality, we have
\begin{equation}
\label{FF}
\|F({\bf X})\|\leq \|F(r{\bf W})\|, \qquad {\bf X}\in  r{\bf D_q}(\cH)^-.
\end{equation}
Since $\cH$ is an infinite dimensional Hilbert space, there is a subspace $\cM\subset \cH$ and a unitary operator $U:\otimes_{i=1}^k F^2(H_{n_i})\to \cM$.
Consider the operators $A_{i,j}:=\left(\begin{matrix} U{\bf W}_{i,j}U^*& 0\\0&0\end{matrix}\right)$, for any $i\in \{1,\ldots, k\}$ and $j\in \{1,\ldots, n_i\}$, with respect to the decomposition $\cH=\cM\oplus \cM^\perp$.
Set ${\bf A}:=(A_1,\ldots, A_k)$ with $A_i:=(A_{i,1},\ldots, A_{i,n_i})$ and note that
${\bf \Delta_{q,A}}(I)= \left(\begin{matrix} UP_\CC U^*& 0\\0&I\end{matrix}\right)$.
Since ${\bf \Delta_{q,A}}(I)\geq 0$ but ${\bf \Delta_{q,A}}(I)$ is not invertible, Proposition \ref{reg-poly} shows that
 ${\bf A}\in \partial {\bf D_q}(\cH)\subset {\bf D_q}(\cH)^-$. Consequently, we have
 $r{\bf A}\in \partial (r{\bf D_q}(\cH))\subset r{\bf D_q}(\cH)^-$ and
 $$
 F(r{\bf A})=\left(\begin{matrix} UF(r{\bf W})U^*& 0\\0& F(0)\end{matrix}\right).
 $$
Hence, we deduce that
$\| F(r{\bf A})\|=\|F(r{\bf W})\|$.
Using now the inequality \eqref{FF},  we complete the proof of the first part of the theorem.

To prove the second part, assume  that  $F$ has a continuous extension $\widetilde{F}$ to ${\bf D_q}(\cK)^-$  in the operator norm, for any Hilbert space $\cK$.
Then $\widetilde{F}({\bf A})=\lim_{r\to 1}F(r{\bf A})$ exists in the operator norm and is equal  to
$$\left(\begin{matrix} U\lim\limits_{r\to 1}F(r{\bf W})U^*& 0\\0& F(0)\end{matrix}\right)
=\left(\begin{matrix} U\widetilde{F}({\bf W})U^*& 0\\0& F(0)\end{matrix}\right).
$$
Hence, $\|\widetilde{F}({\bf A})\|=\lim_{r\to 1}\|F(r{\bf W})\|=\|F\|_\infty$.
Since ${\bf A}\in \partial {\bf D_q}(\cH)\subset {\bf D_q}(\cH)^-$, the proof is complete.
 \end{proof}

For the rest of this section, we assume that $\cH$ is a separable infinite dimensional Hilbert space.
Since  ${\bf D_q}(\cH)$ is a complete Reinhardt domain and
 $$
 B(\cH)^{n_1}\times_c\cdots \times_c B(\cH)^{n_k}=\bigcup_{\rho>0} \rho{\bf D_q}(\cH),
$$
we can define the Minkovski functional associated with the regular polydomain ${\bf D_q}(\cH)$   to be   the function
 $m_{\bf B_n}:B(\cH)^{n_1}\times_c\cdots \times_c B(\cH)^{n_k}\to [0,\infty)$  given by
 $$
 m_{\bf D_q}({\bf X}):=\inf \left\{r>0: \ {\bf X}\in r{\bf D_q}(\cH)\right\}.
 $$
 The polyball ${\bf P_n}$ is defined by setting ${\bf P_n}(\cH):=[B(\cH)^{n_1}]_1\times_c\cdots \times_c [B(\cH)^{n_k}]_1$.

 \begin{proposition}\label{Mink} The  Minkovski functional associated with the regular polydomain ${\bf D_q}(\cH)$ has the following properties:
 \begin{enumerate}
 \item[(i)] $m_{\bf D_q}(\lambda{\bf X})=|\lambda| m_{\bf D_q}({\bf X})$ for $\lambda\in \CC$;

 \item[(ii)] $m_{\bf D_q}$ is upper semicontinuous;
 \item[(iii)]  ${\bf D_q}(\cH)=\{{\bf X}\in B(\cH)^{n_1}\times_c\cdots \times_c B(\cH)^{n_k}: \ m_{\bf D_q}({\bf X})<1\}$;
     \item[(iv)] ${\bf D_q}(\cH)^-=\{{\bf X}\in B(\cH)^{n_1}\times_c\cdots \times_c B(\cH)^{n_k}: \ m_{\bf D_q}({\bf X})\leq 1\}$;
         \item[(v)] There is a polyball $r{\bf P_n}(\cH)\subset{\bf D_q}(\cH)$ for some $r\in (0,1)$, where $m_{\bf D_q}$ is continuous.
 \end{enumerate}

 \end{proposition}

 \begin{proof} Assume that ${\bf X}\in B(\cH)^{n_1}\times_c\cdots \times_c B(\cH)^{n_k}$ and $\lambda\in \CC$ are such that ${\bf X}\neq 0$ and $\lambda\neq 0$. It is clear that
 $m_{\bf D_q}(\lambda{\bf X})=t>0$ if and only if
 $\lambda{\bf X}\in c{\bf D_q}(\cH)$ for any $c>t$, and
 $\lambda{\bf X}\notin d{\bf D_q}(\cH)$ if $0<d<t$. Since
 ${\bf D_q}(\cH)=e^{i\theta}{\bf D_q}(\cH)$ for any $\theta\in \RR$, we deduce that the latter conditions are equivalent to
 ${\bf X}\in \frac{c}{|\lambda|}{\bf D_q}(\cH)$ for any $c>t$ and
  ${\bf X}\notin \frac{d}{|\lambda|}{\bf D_q}(\cH)$  if $0<d<t$. Consequently, we obtain that $m_{\bf D_q}({\bf X})=\frac{t}{|\lambda|}$, which shows that item (i) holds.   Item (ii) follows easily from item  (i).

  Due to Proposition \ref{reg-poly}, we have $ {\bf D_q}(\cH)=\bigcup_{0<r<1}r{\bf D_q}(\cH)$. Consequently, one can easily deduce item (iii). As we saw in the proof of the same proposition, for any $r\in (0,1)$, we have ${\bf D_q}(\cH)^-\subseteq \frac{1}{r}{\bf D_q}(\cH)$. Hence, $m_{\bf D_q}({\bf X})\leq 1$ for any ${\bf X}\in  {\bf D_q}(\cH)^-$.  Assume that ${\bf X}\in B(\cH)^{n_1}\times_c\cdots \times_c B(\cH)^{n_k}$ and $m_{\bf D_q}({\bf X})= 1$. Then there is a sequence $\{t_m\}$ with $t_m>1$ and $t_m\to 1$ such that ${\bf X}\in t_m {\bf D_q}(\cH)$ for any $m\in \NN$. Taking $t_m\to 1$, we deduce that ${\bf X}\in {\bf D_q}(\cH)^-$. Consequently,   using item (iii), one can deduce   item (iv). To prove (v), note that the fact that $r{\bf P_n}(\cH)\subset{\bf D_q}(\cH)$ for some $r\in (0,1)$ is obvoius. The continuity of
 $m_{\bf D_q}$ on $r{\bf P_n}(\cH)$ is due to the convexity of the latter polyball. The proof is complete.
\end{proof}

Now, we present an analogue of Schwarz lemma from complex analysis in the context of free holomorphic functions on polydomains.

 \begin{theorem} \label{Schwarz} Let $F:{\bf D_q}(\cH)\to B(\cH)^p$ be a bounded  free holomorphic function with $\|F\|_\infty\leq 1$. If $F(0)=0$, then
 $$
 \|F({\bf X})\|\leq m_{\bf D_q}({\bf X})<1,\qquad  {\bf X}\in {\bf D_q}(\cH),
 $$
 where $m_{\bf B_n}$ is the Minkovski functional associated with the regular polydomain ${\bf D_q}(\cH)$.
In particular, if $p=1$, the free holomorphic function
 $$
 \psi({\bf X}):=\sum_{i=1}^k\sum_{j=1}^{n_j}\frac {\partial f}{\partial z_{i,j}}(0) X_{i,j},\qquad {\bf X}=\{X_{i,j}\}\in {\bf D_q}(\cH),
 $$
  has the property that $\|\psi({\bf X})\|\leq m_{\bf D_q}({\bf X})<1$ for any $ {\bf X}\in {\bf D_q}(\cH)$,  where $f({\bf z}):=F({\bf z})$ for any ${\bf z}=\{z_{i,j}\}\in {\bf D_q}(\CC)$, is the scalar representation of $F$.
 \end{theorem}
  \begin{proof} Fix ${\bf X}\in {\bf D_q}(\cH)$. Due to Proposition \ref{Mink}, we have $m_{\bf D_q}({\bf X})<1$.  Let $t\in (0,1)$ be such that
  $m_{\bf D_q}({\bf X})<t<1$.
   Hence, using again Proposition \ref{Mink} we deduce that $\frac{1}{t}{\bf X}\in {\bf D_q}(\cH)$ which,  due to   Proposition \ref{reg-poly},  implies   $\frac{\lambda}{t}{\bf X}\in {\bf D_q}(\cH)$ for any $\lambda\in \DD:=\{z\in \CC: \ |z|<1\}$.
  For each $x,y\in \cH^{(p)}:=\cH\oplus \cdots \oplus \cH$ with $\|x\|\leq 1$ and $\|y\|\leq 1$, define the function  $\varphi_{x,y}:\DD\to \CC$ by setting
  $$
  \varphi_{x,y}(\lambda):=\left<F\left(\frac{\lambda}{t}{\bf X}\right)x,y\right>,\qquad \lambda\in \DD.
  $$
Since   $F$ is a free holomorphic function on ${\bf D_q}(\cH)$ and $\|F\|_\infty\leq 1$, we deduce that
$\varphi_{x,y}$ is a holomorphic function on the unit disc $\DD$  and $|\varphi_{x,y}(\lambda)|\leq 1$.
Since $\varphi_{x,y}(0)=0$,  the classical Schwarz lemma   implies
$|\varphi_{x,y}(\lambda)|\leq |\lambda|$ for any $\lambda\in \DD$.
Setting $\lambda=m_{\bf D_q}({\bf X})$, we deduce that
$$
  \varphi_{x,y}(\lambda):=\left<F\left(\frac{m_{\bf D_q}({\bf X})}{t}{\bf X}\right)x,y\right>\leq m_{\bf D_q}({\bf X}),\qquad \lambda\in \DD,
  $$
for any $t\in (0,1)$  with
  $m_{\bf D_q}({\bf X})<t<1$. Using the fact that $F$ is continuous on ${\bf D_q}(\cH)$ and taking $t\to m_{\bf D_q}({\bf X})$, we obtain
  $
  |\left<F({\bf X})x,y\right>|\leq m_{\bf D_q}({\bf X})
$
for any  $x,y\in \cH^{(p)}$ with $\|x\|\leq 1$ and $\|y\|\leq 1$. Hence,
$$\|F({\bf X})\|\leq m_{\bf D_q}({\bf X})<1, \qquad {\bf X}\in {\bf D_q}(\cH).
 $$

   To prove the last part of the theorem, assume that $p=1$.
Due to the classical Schwarz lemma,  we   have $|\varphi_{x,y}'(0)|\leq 1$.
On the other hand,  $\varphi_{x,y}'(0)=\left<\frac{1}{t}\psi({\bf X}) x,y\right>$, which implies $\|\psi({\bf X})\|\leq t<1$. Now, taking $t\to m_{\bf D_q}({\bf X})$, we deduce that
 $\|\psi({\bf X})\|\leq m_{\bf D_q}({\bf X})<1$.
 The proof is complete.
 \end{proof}

 \bigskip

\section{Weierstrass, Montel, Vitali theorems for the algebra $Hol(\bf{D_q})$}
 \label{Weierstrass}

In this section, we obtain Weierstrass, Montel, and Vitali type theorems for
the algebra  $Hol({\bf D_q})$ of free holomorphic functions on the  noncommutative
domain ${\bf D_q}$. This enables us to introduce a metric on
$Hol({\bf D_q})$ with respect to which it becomes a complete metric
space.

The first result   is  an analogue of
   Weierstrass theorem (see \cite{Co}) for free holomorphic functions on noncommutative polydomains.

\begin{theorem}\label{Weierstrass1}
 Let $\{G_m\}_{m=1}^\infty\subset Hol({\bf D_q})$ be a
 sequence of free holomorphic functions   such that, for each
 $r\in[0,1)$, the sequence $\{G_m(r{\bf W}\}_{m=1}^\infty$
 is convergent in the operator norm topology.
  Then there is a free holomorphic function   $F\in Hol({\bf D_q})$
   such that
$G_m(r{\bf W})$ converges to $F(r{\bf W})$, as $m\to \infty$,   for
any  $r\in [0,1)$.
\end{theorem}

\begin{proof} Let $G_m$ have the representation  $G_m:= \sum\limits_{{\bf p}=(p_1,\ldots, p_k)\in \ZZ_+^k}\sum\limits_{\boldsymbol\alpha\in  \Lambda_{\bf p}} c_{(\boldsymbol\alpha)}^{(m)} {\bf Z}_{\boldsymbol\alpha}$ and fix $r\in
(0,1)$.  Since $G_m$ is a free holomorphic function on ${\bf D_q}$, Corollary \ref{Cara}  shows that
$$
G_m(r{\bf W})=\sum\limits_{{\bf p}=(p_1,\ldots, p_k)\in \ZZ_+^k}\sum\limits_{\boldsymbol\alpha\in  \Lambda_{\bf p}} c_{(\boldsymbol\alpha)}^{(m)} r^{p_1+\cdots p_k} {\bf W}_{\boldsymbol\alpha}
$$
is in the noncommutative disc algebra $\boldsymbol\cA({\bf D_q})$. Due to the hypothesis,
  the sequence   $\{G_m(r{\bf W})\}_{m=1}^\infty$ is convergent in the operator norm of
$B(F^2(H_{n_1})\otimes \cdots \otimes F^2(H_{n_k}))$. Consequently,    there exists and operator  $G({\bf W})\in \boldsymbol\cA({\bf D_q})$  such that
\begin{equation}\label{Fm-to}
G_m(r{\bf W})\to G({\bf W}), \quad \text{ as } \
m\to\infty.
\end{equation}
 Let $\sum\limits_{{\bf p}=(p_1,\ldots, p_k)\in \ZZ_+^k}\sum\limits_{\boldsymbol\alpha\in  \Lambda_{\bf p}}
 d_{(\boldsymbol\alpha)}(r) {\bf W}_{\boldsymbol\alpha}$ be the Fourier
  representation
  of $G({\bf W})$, with
  $d_{(\boldsymbol\alpha)}(r)\in \CC$. Since  the operators ${\bf W}_{\boldsymbol\alpha}$, $\boldsymbol\alpha\in \Lambda_{\bf p}$,
   have orthogonal ranges and
   $$
   {\bf W}_{\boldsymbol \alpha}(1)=\frac{1}{\sqrt{b_{1,\alpha_1}\cdots b_{k,\alpha_k}}}e^1_{\alpha_1}\otimes \cdots \otimes e^k_{\alpha_k},
   $$
    we deduce that
\begin{equation}
\label{d_alpha}
 d_{(\boldsymbol\alpha)}(r)\frac{1}{{b_{1,\alpha_1}\cdots b_{k,\alpha_k}}}=\left< {\bf W}_{\boldsymbol \alpha}^*
G({\bf W})1,1\right>, \quad \boldsymbol\alpha=(\alpha_1,\ldots, \alpha_k)\in \FF_{n_1}^+\times \cdots \times\FF_{n_k}^+.
\end{equation}
If $\lambda_{(\boldsymbol\beta)}\in \CC$ for    $\boldsymbol\beta=(\beta_1,\ldots, \beta_k)\in  \FF_{n_1}^+\times \cdots \times\FF_{n_k}^+$ with
$|\beta_i|=p_i$, $i\in \{1,\ldots, k\}$,  we have
\begin{equation*}
\begin{split}
\Bigl|\Bigl<\sum\limits_{\boldsymbol\beta\in  \Lambda_{\bf p}}  &\lambda_{(\boldsymbol\beta)}
{\bf W}_{\boldsymbol\beta}^*(G_m(r{\bf W})  -G({\bf W}))1,1\Bigr>\Bigr| \\
&\leq \|G_m(r{\bf W})-G({\bf W})\|
\left\|\sum\limits_{\boldsymbol\beta\in  \Lambda_{\bf p}}  \lambda_{(\boldsymbol\beta)}
{\bf W}_{\boldsymbol\beta}^*\right\|.
\end{split}
\end{equation*}
 Consequently, due to relation \eqref{d_alpha} and  Lemma \ref{Le}, we have
\begin{equation*}
\begin{split}
&\left|\sum\limits_{\boldsymbol\beta\in  \Lambda_{\bf p}}\frac{1}{ {b_{1,\beta_1}\cdots b_{k,\beta_k}}}
(r^{p_1+\cdots + p_k}c_{(\boldsymbol\beta)}^{(m)}-d_{(\boldsymbol\beta)}(r))\lambda_{(\boldsymbol\beta)}\right|\\
& \qquad\leq \|G_m(r{\bf G})-G({\bf G})\|
\left(\sum\limits_{\boldsymbol\beta\in  \Lambda_{\bf p}}
\frac{1}{ {b_{1,\beta_1}\cdots b_{k,\beta_k}}}|\lambda_{(\boldsymbol\beta)}|^2\right)^{1/2}
\end{split}
\end{equation*}
for any $\lambda_{(\boldsymbol\beta)}\in \CC$ with    $\boldsymbol\beta=(\beta_1,\ldots, \beta_k)\in  \FF_{n_1}^+\times \cdots \times\FF_{n_k}^+$ and
$|\beta_i|=p_i$.
Consequently,
$$
\left(\sum\limits_{\boldsymbol\beta\in  \Lambda_{\bf p}}\frac{1}{{b_{1,\beta_1}\cdots b_{k,\beta_k}}}
|r^{p_1+\cdots + p_k}c_{(\boldsymbol\beta)}^{(m)}-d_{(\boldsymbol\beta)}(r)|^2\right)^{1/2}\leq
\|G_m(r{\bf W})-G({\bf W})\|
$$
for any ${\bf p}:=(p_1,\ldots, p_k)\in \ZZ_+^k$. Now, since $\|G_m(r{\bf W})-G({\bf W})\| \to 0$, as $m\to\infty$, the inequality
above  implies  $r^{p_1+\cdots + p_k}c_{(\boldsymbol\beta)}^{(m)}\to d_{(\boldsymbol\beta)}(r)$, as $m\to\infty$,
for any $\boldsymbol\beta=(\beta_1,\ldots, \beta_k)\in  \FF_{n_1}^+\times \cdots \times\FF_{n_k}^+$ with
$|\beta_i|=p_i$ and any ${\bf p}\in \ZZ_+^k$. Therefore,
$c_{(\boldsymbol\beta)}:=\lim\limits_{m\to\infty}c_{\boldsymbol\beta)}^{(m)}$ exists and
$d_{(\boldsymbol\beta)}(r)=r^{p_1+\cdots + p_k} c_{(\boldsymbol\beta)}$. Define the  formal power series $F:=
\sum\limits_{\boldsymbol\alpha\in \FF_{n_1}^+\times \cdots \times\FF_{n_k}^+} c_{(\boldsymbol\alpha)} {\bf Z}_{\boldsymbol\alpha}$ and let us prove  that
$F$ is a free holomorphic function on  the noncommutative polydomain
${\bf D_q}$.  Using the inequality above, the fact that
$\left\|\sum\limits_{\boldsymbol\beta\in  \Lambda_{\bf p}}
b_{1,\beta_1}\cdots b_{k,\beta_k}
{\bf W}_{\boldsymbol\beta} {\bf W}_{\boldsymbol\beta}^* \right\|=1$,  and  Lemma \ref{Le}, we deduce that
\begin{equation*}
\begin{split}
 &\left\|\sum\limits_{\boldsymbol\alpha\in  \Lambda_{\bf p}} c_{(\boldsymbol\alpha)}^{(m)} r^{p_1+\cdots p_k} {\bf W}_{\boldsymbol\alpha}
 -\sum\limits_{\boldsymbol\alpha\in  \Lambda_{\bf p}} c_{(\boldsymbol\alpha)} r^{p_1+\cdots p_k} {\bf W}_{\boldsymbol\alpha}\right\|\\
 &=r^{p_1+\cdots +p_k}
 \left\|\sum\limits_{\boldsymbol\alpha\in  \Lambda_{\bf p}}\left(c_{(\boldsymbol\alpha)}^{(m)}-c_{(\boldsymbol\alpha)}\right)
  {\bf W}_{\boldsymbol\alpha}\right\|\\
  &\leq r^{p_1+\cdots +p_k}
  \left(\sum\limits_{\boldsymbol\alpha\in  \Lambda_{\bf p}}\frac{1}{b_{1,\alpha_1}\cdots b_{k,\alpha_k}}
\left|c_{(\boldsymbol\alpha)}^{(m)}-c_{(\boldsymbol\alpha)}\right|^2\right)^{1/2}
\left\|\sum\limits_{\boldsymbol\beta\in  \Lambda_{\bf p}}
b_{1,\beta_1}\cdots b_{k,\beta_k}
{\bf W}_{\boldsymbol\beta} {\bf W}_{\boldsymbol\beta}^* \right\|^{1/2}.
 \\
&\leq \|G_m(r{\bf W})-G({\bf W})\|
\end{split}
\end{equation*}
for any  $r\in [0,1)$ and any ${\bf p}:=(p_1,\ldots, p_k)\in \ZZ_+^k$.
This shows that
\begin{equation*}
\left\|\sum\limits_{\boldsymbol\alpha\in  \Lambda_{\bf p}} c_{(\boldsymbol\alpha)}^{(m)}  {\bf W}_{\boldsymbol\alpha}\right\|
 \to \left\|\sum\limits_{\boldsymbol\alpha\in  \Lambda_{\bf p}} c_{(\boldsymbol\alpha)}  {\bf W}_{\boldsymbol\alpha}\right\|,\quad \text{ as } \
m\to\infty,
\end{equation*}
uniformly with respect to ${\bf p}\in \ZZ_+^k$. Since  the operators ${\bf W}_{\boldsymbol\alpha}$, $\boldsymbol\alpha\in \Lambda_{\bf p}$,
   have orthogonal ranges, the later convergence is equivalent to
   \begin{equation}
\label{conv-coef}\left\|\sum\limits_{\boldsymbol\alpha\in  \Lambda_{\bf p}} \left|c_{(\boldsymbol\alpha)}^{(m)}\right|^2  {\bf W}_{\boldsymbol\alpha}^* {\bf W}_{\boldsymbol\alpha}\right\|^{1/2}
 \to \left\|\sum\limits_{\boldsymbol\alpha\in  \Lambda_{\bf p}} \left|c_{(\boldsymbol\alpha)}\right|^2  {\bf W}_{\boldsymbol\alpha}^*{\bf W}_{\boldsymbol\alpha}\right\|^{1/2},\quad \text{ as } \
m\to\infty,
\end{equation}
uniformly with respect to ${\bf p}\in \ZZ_+^k$.
 Assume now  that
$\gamma>1$ and
$$
\limsup_{{\bf p}=(p_1,\ldots, p_k)\in \ZZ_+^k}
\left\|\sum\limits_{\boldsymbol\alpha\in  \Lambda_{\bf p}} \left|c_{(\boldsymbol\alpha)}\right|^2  {\bf W}_{\boldsymbol\alpha}^*{\bf W}_{\boldsymbol\alpha}\right\|^{1/2(p_1+\cdots +p_k)}>\gamma.
$$
Then,  there are infinitely many ${\bf p}=(p_1,\ldots, p_k)\in \ZZ_+^k$  such that
\begin{equation}
\label{sup-ga}
\left\|\sum\limits_{\boldsymbol\alpha\in  \Lambda_{\bf p}} \left|c_{(\boldsymbol\alpha)}\right|^2  {\bf W}_{\boldsymbol\alpha}^*{\bf W}_{\boldsymbol\alpha}\right\|^{1/2}>\gamma^{p_1+\cdots +p_k}.
\end{equation}
Let $\lambda$  be such that $1<\lambda< \gamma$ and let $\epsilon>0$
be with the property that $\epsilon<\gamma-\lambda$. Note
that $\epsilon<\gamma^{p_1+\cdots +p_k}-\lambda^{p_1+\cdots +p_k}$ for any ${\bf p}=(p_1,\ldots, p_k)\in \ZZ_+^k$. The
convergence in  relation \eqref{conv-coef} implies that there exists
$K_\epsilon\in \NN$ such that
$$
\left|\left\|\sum\limits_{\boldsymbol\alpha\in  \Lambda_{\bf p}} \left|c_{(\boldsymbol\alpha)}^{(m)}\right|^2  {\bf W}_{\boldsymbol\alpha}^* {\bf W}_{\boldsymbol\alpha}\right\|^{1/2}
  -\left\|\sum\limits_{\boldsymbol\alpha\in  \Lambda_{\bf p}} \left|c_{(\boldsymbol\alpha)}\right|^2  {\bf W}_{\boldsymbol\alpha}^*{\bf W}_{\boldsymbol\alpha}\right\|^{1/2}\right|<
\epsilon
$$
for any  $m>K_\epsilon$ and any ${\bf p}=(p_1,\ldots, p_k)\in \ZZ_+^k$. Consequently, if we fix $m>K_\epsilon$ and  using
  \eqref{sup-ga}, we deduce that
$$
\left\|\sum\limits_{\boldsymbol\alpha\in  \Lambda_{\bf p}} \left|c_{(\boldsymbol\alpha)}^{(m)}\right|^2  {\bf W}_{\boldsymbol\alpha}^* {\bf W}_{\boldsymbol\alpha}\right\|^{1/2}\geq
\gamma^{p_1+\cdots +p_k}-\epsilon>\lambda^{p_1+\cdots +p_k}
$$
for   infinitely many ${\bf p}=(p_1,\ldots, p_k)\in \ZZ_+^k$.  Therefore,
$$
\limsup_{{\bf p}=(p_1,\ldots, p_k)\in \ZZ_+^k}
\left\|\sum\limits_{\boldsymbol\alpha\in  \Lambda_{\bf p}} \left|c_{(\boldsymbol\alpha)}^{(m)}\right|^2  {\bf W}_{\boldsymbol\alpha}^* {\bf W}_{\boldsymbol\alpha}\right\|^{1/2(p_1+\cdots +p_k)}\geq
\lambda>1
$$
 Now, using  Theorem \ref{Hadamard} and Lemma \ref{Le}, we conclude that $G_m$  is not a free
holomorphic function  on ${\bf D_q}$, which
is a contradiction. Therefore, we must have
$$
\limsup_{{\bf p}=(p_1,\ldots, p_k)\in \ZZ_+^k}
\left\|\sum\limits_{\boldsymbol\alpha\in  \Lambda_{\bf p}} \left|c_{(\boldsymbol\alpha)}\right|^2  {\bf W}_{\boldsymbol\alpha}^*{\bf W}_{\boldsymbol\alpha}\right\|^{1/2(p_1+\cdots +p_k)}\leq 1.
$$
Using again  Theorem \ref{Hadamard}  and Lemma \ref{Le}, we deduce that  $F$ is a
free holomorphic function on  ${\bf D_q}$. Consequently,
$$F(r{\bf W})=\sum\limits_{{\bf p}=(p_1,\ldots, p_k)\in \ZZ_+^k}\sum\limits_{\boldsymbol\alpha\in  \Lambda_{\bf p}} c_{(\boldsymbol\alpha)} r^{p_1+\cdots p_k}{\bf W}_{\boldsymbol\alpha}$$
 is convergent in the operator norm topology and $G({\bf W})=F(r{\bf W})$. Due to relation
\eqref{Fm-to}, for each $r\in [0,1)$,  we have
$$
\|G_m(r{\bf W})-F(r{\bf W})\|\to 0, \quad \text{
as }\ m\to \infty.
$$
  The proof is
complete.
 \end{proof}

We say that  $\cG\subset Hol({\bf D_q})$ is a {\it normal set } if each sequence
$\{G_m\}_{m=1}^\infty$  in $\cG$ has a subsequence
$\{G_{m_k}\}_{k=1}^\infty$ which converges to an element $G\in Hol({\bf D_q})$,
i.e.,   for any $r\in [0,1)$,
$$
\|G_{m_k}(r{\bf W})-F(r{\bf W})\|\to 0, \quad
\text{ as }\ k\to \infty.
$$
 The set $\cG$ is called  {\it locally bounded} if, for any $r\in[0,1)$,
 there exists $M>0$ such that
$\|F(r{\bf W})\|\leq M$ for all $F\in \cG$.

An important consequence of Theorem \ref{Weierstrass1} is   the
following noncommutative version of Montel theorem (see \cite{Co}).

\begin{theorem}\label{Montel}
Let $\cG\subset Hol({\bf D_q})$  be a family of free holomorphic
functions. Then the  following statements are equivalent:
\begin{enumerate}
\item[(i)] $\cG$ is  locally bounded.
\item[(ii)]  $\cG$ is a  normal set.
\end{enumerate}
\end{theorem}
\begin{proof}
First, we prove the implication (i)$\implies$(ii). Assume that  $\cG$ is  locally bounded., i.e., for each $r\in [0,1)$, there is $M_r>0$ such that
$$
\sup_{G\in \cG} \|G(r{\bf W})\|=M_r<\infty.
$$
Each function $G\in \cG$ has a representation
$$
G({\bf X})=\sum\limits_{{\bf p}=(p_1,\ldots, p_k)\in \ZZ_+^k}\sum\limits_{\boldsymbol\alpha\in  \Lambda_{\bf p}} c_{(\boldsymbol\alpha)}^G {\bf X}_{\boldsymbol\alpha},
$$
 where $c_{(\boldsymbol\alpha)}^G\in \CC$.
 Let $\{G_m\}_{m=1}^\infty$  be a sequence of elements in $\cG$.
 Since $| c_{(\boldsymbol g)}^G|=\|G(0)\|\leq M_0$ for any $G\in \cG$, the classical Bolzano-Weierstrass theorem for bounded sequences of complex numbers shows that there is a subsequence $\{G_{\gamma_k}\}_{k=1}^\infty$ of  $\{G_m\}_{m=1}^\infty$ such that $\{c_{(\boldsymbol g)}^{G_{\gamma_k}}\}$ is convergent in $\CC$.
  On the other hand, due to Theorem \ref{Cauchy-ineq}, for each ${\bf p}=(p_1,\ldots, p_k)\in \ZZ_+^k$ and $r\in (0,1)$, we have
 \begin{equation}
 \label{Cau-estim}
\left(\sum_{\boldsymbol\alpha=(\alpha_1,\ldots, \alpha_k)\in  \Lambda_{\bf p}} \frac{1}{{b_{1,\alpha_1}\cdots b_{k,\alpha_k}}}| c_{(\boldsymbol g)}^G|^2\right)^{1/2}\leq \frac{1}{r^{p_1+\cdots +p_k}}\|G(r{\bf W})\|\leq \frac{1}{r^{p_1+\cdots +p_k}}M_r
\end{equation}
for any $G\in \cG$.
Using this inequality, the classical Bolzano-Weierstrass theorem, and the diagonal process, an inductive argument shows that there is a subsequence
$\{G_{m_s}\}_{s=1}^\infty$ of $\{G_m\}_{m=1}^\infty$ such that, for each
$\boldsymbol\alpha=(\alpha_1,\ldots, \alpha_k)\in \FF_{n_1}^+\times \cdots \times \FF_{n_k}^+$,  the sequence $\{c_{(\boldsymbol \alpha)}^{G_{m_s}}\}_{s=1}^\infty$ is convergent in $\CC$, as $m_s\to \infty$.

Fix $r\in (0,1)$ and let $t>1$. We prove that the sequence $\{G_{m_s}(\frac{r}{t}{\bf W})\}_{s=1}^\infty$ is convergent in the operator norm topology of $B(F^2(H_{n_1})\otimes \cdots \otimes F^2(H_{n_k}))$.
For each $N, s, \ell\in \NN$, set
$$
\Lambda_N(s,\ell):=\sum_{m=0}^N\sum_{{{\bf p}:=(p_1,\ldots, p_k)\in\ZZ_+^k}\atop{p_1+\cdots +p_k=m}}
 \left\|\sum_{\boldsymbol\alpha\in \Lambda_{\bf p}}
 c_{(\boldsymbol \alpha)}^{G_{m_s}}\left(\frac{r}{t}\right)^{p_1+\cdots +p_k} {\bf W}_{\boldsymbol \alpha} - c_{(\boldsymbol \alpha)}^{G_{m_\ell}}\left(\frac{r}{t}\right)^{p_1+\cdots +p_k} {\bf W}_{\boldsymbol \alpha} \right\|
 $$
and note that relation \eqref{Cau-estim} implies
\begin{equation*}
\begin{split}
&\left\|G_{m_s}\left(\frac{r}{t}{\bf W}\right)-G_{m_\ell}\left(\frac{r}{t}{\bf W}\right)\right\|\\
&\leq \Lambda_N(s,\ell)
+\sum_{m=N+1}^\infty \sum_{{{\bf p}:=(p_1,\ldots, p_k)\in\ZZ_+^k}\atop{p_1+\cdots +p_k=m}} \left(\frac{r}{t}\right)^m
\left(\sum_{\boldsymbol\alpha=(\alpha_1,\ldots, \alpha_k)\in  \Lambda_{\bf p}} \frac{1}{{b_{1,\alpha_1}\cdots b_{k,\alpha_k}}}| c_{(\boldsymbol g)}^{G_{m_s}}- c_{(\boldsymbol g)}^{G_{m_\ell}}|^2\right)^{1/2}\\
&\leq
 \Lambda_N(s,\ell)
+\sum_{m=N+1}^\infty \sum_{{{\bf p}:=(p_1,\ldots, p_k)\in\ZZ_+^k}\atop{p_1+\cdots +p_k=m}} \left(\frac{r}{t}\right)^m \frac{2}{r^m} M_r\\
&\leq \Lambda_N(s,\ell)
+\sum_{m=N+1}^\infty   \left(\begin{matrix} m+k-1\\k-1\end{matrix}\right) \frac{2}{t^m}M_r.
\end{split}
\end{equation*}
Now, it is clear that we can choose $N$ large enough so that
$$
\sum_{m=N+1}^\infty   \left(\begin{matrix} m+k-1\\k-1\end{matrix}\right) \frac{2}{t^m}M_r<\frac{\epsilon}{2}.
$$
Since,  for each
$\boldsymbol\alpha=(\alpha_1,\ldots, \alpha_k)\in \FF_{n_1}^+\times \cdots \times \FF_{n_k}^+$,  the sequence $\{c_{(\boldsymbol \alpha)}^{G_{m_s}}\}_{s=1}^\infty$ is convergent in $\CC$, as $m_s\to \infty$, we can choose $k_0\in \NN$ such that
$\Lambda_N(s,\ell)<\frac{\epsilon}{2}$ for any $s,\ell\geq k_0$.
Putting together these results we conclude that
$\left\|G_{m_s}\left(\frac{r}{t}{\bf W}\right)-G_{m_\ell}\left(\frac{r}{t}{\bf W}\right)\right\|\leq \epsilon$ for any $s,\ell\geq k_0$, which proves
that the sequence $\{G_{m_s}(\frac{r}{t}{\bf W})\}_{s=1}^\infty$ is convergent in the operator norm topology. Since $\{\frac{r}{t}:\ r\in [0,1), t>1\}=[0,1)$, we deduce that, for each $t\in [0,1)$, $\{G_{m_s}({t}{\bf W})\}_{s=1}^\infty$ is convergent in the operator norm topology, as $m_s\to \infty$.
Applying Theorem \ref{Weierstrass1}, we conclude that $\cG$ is a normal set  and the implication (i)$\implies$(ii) is true.

We prove the converse by contradiction. Assume that $\cG$ is a normal set and that there is $r_0\in (0,1)$ such that $\sup_{G\in \cG} \|G(r_0{\bf W})\|=\infty$. Then, there is a sequence  $\{G_m\}_{m=1}^\infty\subset \cG$ such that $\|G_m(r_0{\bf W})\|\to \infty$, as $m\to \infty$.
Since $\cG$ is a normal set, there is a subsequence $\{G_{m_k}\}_{k=1}^\infty$ and $G\in Hol({\bf D_q})$ such that $\|G_{m_k}(r{\bf W})-G(r{\bf W})\|\to 0$, as $k\to \infty$, for any $r\in [0,1)$. This contradicts the fact that $\|G_{m_k}(r_0 {\bf W})\|\to \infty$, as $m_k\to \infty$. The proof is complete.
\end{proof}

The next result is an analogue of  Vitali's theorem in our setting.

\begin{theorem} Let $\{G_m\}_{m=1}^\infty$ be a sequence of free holomorphic function (with scalar coefficients) on the polydomain ${\bf D_q}$ such that , for each $r\in [0,1)$,
$$\sup_{m}\|G_m(r{\bf W})\|<\infty.
$$
If there is $\gamma\in (0,1)$ such that $G_m(\gamma {\bf W})$ converges in the operator norm, as $m\to \infty$, then there is $G\in Hol({\bf D_q})$ such that
$$
\|G_m(r{\bf W})-G(r{\bf W})\|\to 0, \quad \text{ as } \ m\to \infty.
$$
for any $r\in [0,1)$.
\end{theorem}
\begin{proof}
Assume that there is $r_0\in [0,1)$ such that $\{G_m(r_0{\bf W})\}_{m=1}^\infty$  is not convergent in the operator norm. Then there is $\delta>0$ and subsequences
$\{G_{m_k}\}_{k=1}^\infty$  and $\{G_{s_k}\}_{k=1}^\infty$  of $\{G_m\}_{m=1}^\infty$  such that
\begin{equation}
\label{GG}
\|G_{m_k}(r_0 {\bf W})-G_{s_k}(r_0 {\bf W})\|\geq \delta
\end{equation}
for any $k\in \NN$. Due to Theorem \ref{Montel}, there is $F\in Hol({\bf D_q})$ and a subsequence $\{k_p\}_{p=1}^\infty$ of  $\{k\}_{k=1}^\infty$  such that
\begin{equation}
\label{GF}
G_{m_{k_p}}(r{\bf W})\to  F(r{\bf W}), \quad \text{ as }  p\to \infty,
\end{equation}
for any $r\in [0,1)$.
Note that relation \eqref{GG} implies
\begin{equation}
\label{GG2}
\|G_{m_{k_p}}(r_0 {\bf W})-G_{s_{k_p}}(r_0 {\bf W})\|\geq \delta
\end{equation}
for any $p\in \NN$. Applying again  Theorem \ref{Montel}, there is $H\in Hol({\bf D_q})$ and a subsequence $\{p_q\}_{q=1}^\infty$ of  $\{p\}_{p=1}^\infty$  such that
\begin{equation}
\label{GH}
G_{s_{k_{p_q}}}(r{\bf W})\to  H(r{\bf W}), \quad \text{ as }  p\to \infty,
\end{equation}
for any $r\in [0,1)$. Due to relation \eqref{GG2}, we have
$$
\|G_{m_{k_{p_q}}}(r_0 {\bf W})-G_{s_{k_{p_q}}}(r_0 {\bf W})\|\geq \delta
$$
for any $q\in \NN$. Taking $q\to \infty$ in the latter inequality and using relations \eqref{GF} and \eqref{GH}, we obtain
\begin{equation}
\label{FH}
\|F(r_0{\bf W})-H(r_0{\bf W})\|\geq \delta.
\end{equation}
On the other hand, since the sequence $\{G_m(\gamma {\bf W})\}_{m=1}^\infty$ converges in the operator norm, as $m\to\infty$, relations \eqref{GF} and \eqref{GH} imply $F(\gamma {\bf W})=H(\gamma{\bf W})$. Since $\gamma\in (0,1)$ and $F,H\in Hol({\bf D_q})$,  and due to the uniqueness of the representation of free holomorphic functions on polydomains, we deduce that $F=H$, which contradicts \eqref{FH}. Therefore, for each $r\in [0,1)$, $\{G_m(r{\bf W})\}_{m=1}^\infty$  is convergent in the operator norm, as $m\to \infty$. Applying Theorem \ref{Weierstrass1}, we complete the proof.
\end{proof}

  If $F,G\in Hol( {\bf D_q})$ and
$0<r<1$, we define
$$
d_r(F,G):=\|F(r{\bf W})-G(r{\bf W})\|.
$$
Due to the maximum principle of Theorem \ref{max-mod1}, if  $\cH$ is
an infinite dimensional Hilbert space,  then
$$
d_r(F,G)=\sup_{{\bf X}\in r{\bf D_q}(\cH)^-}
\|F({\bf X})-G({\bf X})\|.
$$
 Let $0<r_m<1$ be such that $\{r_m\}_{m=1}^\infty$ is an increasing sequence
  convergent  to $1$.
For any $F,G\in Hol({\bf D_q})$, we define
$$
\rho (F,G):=\sum_{m=1}^\infty \left(\frac{1}{2}\right)^m
\frac{d_{r_m}(F,G)}{1+d_{r_m}(F,G)}.
$$

  Using standards arguments,
    one can    show  that $\rho$ is a metric
on $Hol({\bf D_q})$.

\begin{theorem}\label{complete-metric}
$\left(Hol({\bf D_q}), \rho\right)$  is a complete metric space.
 \end{theorem}

\begin{proof}

First, note  that  if $\epsilon>0$, then  there exists $\delta>0$
and $m\in \NN$ such that, for any   $F,G\in Hol({\bf D_q})$,
 \ $d_{r_m}(F,G)<\delta\implies \rho(F,G)<\epsilon$.
Conversely, if $\delta>0$ and $m\in \NN$ are fixed, then there is
$\epsilon>0$ such that, for  any $F,G\in Hol({\bf D_q})$,  \ $
\rho(F,G)<\epsilon \implies d_{r_m}(F,G) <\delta.
$

 Now, let
$\{g_k\}_{k=1}^\infty\subset Hol({\bf D_q})$ be a Cauchy sequence in the
metric $\rho$. An immediate consequence of the  observation above  is
that
  $\{g_k(r_m{\bf W})\}_{k=1}^\infty $  is a Cauchy sequence
   in $B(F^2(H_{n_1})\otimes \cdots \otimes F^2(H_{n_k}))$,
   for any
  $m\in \NN$. Consequently, for each $m\in \NN$,
the sequence $\{g_k(r_m{\bf W})\}_{k=1}^\infty$ is
convergent in the operator norm. According to
Theorem \ref{Weierstrass1}, there is a free holomorphic function
$g\in Hol({\bf D_q})$
   such that
$g_k(r{\bf W})$ converges to $g(r{\bf W})$  for
any $r\in [0,1)$. Using again the observation made at the beginning
of this proof, we deduce that $\rho(g_k, g)\to 0$, as $k\to\infty$,
which completes the proof.
 \end{proof}

We remark that, Theorem \ref{Montel} and Theorem \ref{complete-metric} imply the following
compactness criterion for subsets of $Hol({\bf D_q})$.

\begin{corollary}
A subset ~$\cG$ of  $(Hol({\bf D_q}), \rho)$ is compact if and only if
it is closed and locally bounded.
\end{corollary}

\bigskip

\bigskip

       %

      \end{document}